\documentclass[10pt,oneside,reqno]{amsart}
\usepackage{hyperref}
\usepackage[nonewpage]{imakeidx}
\makeindex[columns=3, title=Index of symbols]
\usepackage{seealso}
\usepackage[all,pdf]{xy}
\usepackage{amsmath}
\usepackage{amssymb}
\usepackage{amsxtra}
\usepackage{amsthm}

\allowdisplaybreaks[4]
\usepackage{stmaryrd}
\usepackage{bm}
\usepackage{bbm}
\usepackage{mathrsfs}
\usepackage{upgreek}
\usepackage{thmtools}
\usepackage{mathtools}
\usepackage{geometry}
\usepackage{enumitem}
\usepackage{comment}
\usepackage{microtype}
\usepackage{tikz-cd}

\usepackage{mathpazo}
\usepackage{domitian}
\usepackage[T1]{fontenc}

\AtBeginDocument{
 \DeclareSymbolFont{AMSb}{U}{msb}{m}{n}
 \DeclareSymbolFontAlphabet{\mathbb}{AMSb}
 }
 
\let\SSec\S

\newcommand\rmd{\mathrm{d}}
\newcommand\rme{\mathrm{e}}
\newcommand\rmi{\mathrm{i}}

\newcommand\cC{\mathcal{C}}

\newcommand\cH{\mathcal{H}}

\newcommand\cK{\mathcal{K}}
\newcommand\cL{\mathcal{L}}

\newcommand\cN{\mathcal{N}}
\newcommand\cO{\mathcal{O}}

\newcommand\cX{\mathcal{X}}

\newcommand\bP{\mathbf{P}}

\renewcommand\AA{\mathbb{A}}

\newcommand\CC{\mathbb{C}}

\newcommand\QQ{\mathbb{Q}}
\newcommand\RR{\mathbb{R}}
\renewcommand\SS{\mathbb{S}}

\newcommand\ZZ{\mathbb{Z}}

\newcommand\mf[1]{\mathfrak{#1}}

\newcommand\A{\mathsf{A}}
\newcommand\B{\mathsf{B}}

\newcommand\G{\mathsf{G}}
\renewcommand\H{\mathsf{H}}

\renewcommand\L{\mathsf{L}}
\newcommand\M{\mathsf{M}}

\renewcommand\S{\mathsf{S}}
\newcommand\T{\mathsf{T}}

\newcommand\Z{\mathsf{Z}}

\newcommand\SL{\mathsf{SL}}
\newcommand\GL{\mathsf{GL}}

\newcommand\triv{\mathbbm{1}}
\newcommand\dpii{2\uppi\rmi}
\newcommand\bs{\backslash}

\renewcommand\Re{\mathop{\mathrm{Re}}}

\DeclareMathOperator\Tr{Tr}

\DeclareMathOperator\res{res}

\DeclareMathOperator\ord{ord}

\DeclareMathOperator\Ind{Ind}

\newcommand\el{\mathrm{ell}}
\newcommand\hyp{\mathrm{hyp}}

\newcommand\disc{\mathrm{disc}}
\newcommand\cusp{\mathrm{cusp}}

\newcommand\reg{\mathrm{reg}}

\DeclareMathOperator\vol{vol}
\DeclareMathOperator\orb{orb}
\DeclareMathOperator\worb{worb}
\renewcommand\leq{\leqslant}
\renewcommand\geq{\geqslant}
\let\oldsslash\sslash
\renewcommand\sslash{{\oldsslash}}

\newcommand\legendresymbol[2]{\genfrac{(}{)}{}{}{#1}{#2}}

\makeatletter
\def\subsection{\@startsection{subsection}{2}
  \z@{3pt\@plus0pt}{-.5em}%
  {\normalfont\bfseries}}
\makeatother

\makeatletter
\def\subsubsection{\@startsection{subsubsection}{2}
  \z@{0pt\@plus0pt}{-.5em}%
  {\normalfont\itshape}}
\makeatother

\makeatletter
\def\@seccntformat#1{%
  \protect\textup{\protect\@secnumfont
    \ifnum\pdfstrcmp{subsection}{#1}=0 \bfseries\fi
    \csname the#1\endcsname
    \protect\@secnumpunct
  }%
}  
\makeatother

\newtheoremstyle{THEOREM}{2.5pt}{2pt}{\itshape}{}{\bfseries}{.}{.5em}{\thmname{#1}\thmnumber{ #2}\thmnote{ (#3)}}

\newtheoremstyle{DEFINITION}{2.5pt}{2pt}{}{}{\bfseries}{.}{.5em}{\thmname{#1}\thmnumber{ #2}\thmnote{ (#3)}}

\newtheoremstyle{EXERCISE}{2pt}{2pt}{}{}{\scshape}{.}{.5em}{\thmname{#1}\thmnumber{ #2}\thmnote{ (#3)}}

\theoremstyle{THEOREM}
\newtheorem{theorem}{Theorem}[section]
\newtheorem{lemma}[theorem]{Lemma}
\newtheorem{proposition}[theorem]{Proposition}

\theoremstyle{DEFINITION}

\theoremstyle{EXERCISE}
\newtheorem{remark}[theorem]{Remark}

\numberwithin{equation}{section}

\makeatletter
\renewenvironment{proof}[1][\proofname]{\par
  \vspace{-6pt}
  \pushQED{\qed}
  \normalfont \topsep6\p@\@plus6\p@\relax
  \trivlist
  \item[\hskip\labelsep\rmfamily\bfseries
    #1\@addpunct{:}]\ignorespaces
}{
  \popQED\endtrivlist\@endpefalse
  \vspace{-6pt}
}
\makeatother



\begin{document}
\setlength \lineskip{3pt}
\setlength \lineskiplimit{3pt}
\setlength \parskip{1pt}
\setlength \partopsep{0pt}

\title{Beyond endoscopy for $\GL_2$ over $\QQ$ with ramification 5: cancellation theory} 
\author{Yuhao Cheng}
\address{Qiuzhen College, Tsinghua University, 100084, Beijing, China}
\email{chengyuhaomath@gmail.com}
\keywords{beyond endoscopy, trace formula, Arthur's kernel function, cancellation theory}
\subjclass[2020]{Primary 11F70; 11F72; Secondary 22E30; 22E35}
\date{16 August, 2026}
\begin{abstract}
We complete our work on $\GL_2$ over $\QQ$ in the ramified setting for \emph{Beyond Endoscopy} proposed by Langlands. We prove that the asymptotic formula for each term of the trace formula when summing over $n<X$ with arbitrary smooth test functions at places in $S=\{\infty,q_1,\dots q_r\}$ with $2\in S$, for the standard representation, is $o(X)$. We prove an identity with a variable $X$, called the \emph{limit form of the trace formula} for $\GL_2$ over $\QQ$, directly. The proof uses Arthur's result on the Fourier transform of weighted orbital integrals to rewrite the term involving intertwining operators, and then compares the expansion with the results of the real case due to Arthur-Herb-Sally and Hoffmann, and the nonarchimedean case by direct computation using Arthur's definition.
\end{abstract}

\maketitle
 
\tableofcontents

\section{Introduction}
\subsection{A new way to detect functoriality}
Let $\G,\H$ be reductive algebraic groups over the field $\QQ$ of rational numbers such that $\G$ is quasi-split. Let $\AA$ denote the ad\`ele ring of $\QQ$. Let $^L\G$ be the $L$-group of $\G$.
Conjecturally any homomorphism $\phi\colon ^L\H\to {}^L\G$ gives a functorial lift from the automorphic representations $\sigma$ of $\H$ to the automorphic representations $\phi_*\sigma$ of $\G$. 
A natural question is to ask the converse, that is, how to determine whether an automorphic representation $\pi$ of $\G$ is given by $\phi_*\sigma$ for some automorphic representation $\sigma$ of $\H$, and how to find the smallest such $\H$ (which we call \emph{primitive}). More precisely, given an automorphic representation $\pi$ of $\G$ (corresponding to a map $\cL\to \prescript{L}{}\G$ by the expected Langlands correspondence, where $\cL$ denotes the $L$-packet) and a morphism $\phi\colon\prescript{L}{}\H\to \prescript{L}{}\G$,  we want to determine whether there is a map $\cL\to \prescript{L}{}\H$ making the diagram
\[
      \xymatrix
      {
         & \ar@{-->}[ld]\mathcal{L}\ar^{\pi}[rd] &  \\
      \prescript{L}{}\H \ar_{\phi}[rr]&            & \prescript{L}{}\G 
      }
\]
commute.

One way to study this is to consider the order of $L$-functions. Let $\rho$ be a complex representation of $^L\G$. For a finite set $S$ of places of $\QQ$ containing the archimedean place $\infty$, the partial $L$-function $L^{S}(s,\pi,\rho)$ can be defined and expressed as a Dirichlet series
\[
L^{S}(s,\pi,\rho)=\sum_{\substack{n=1\\ \gcd(n,S)=1}}^{+\infty}\frac{a_{\pi,\rho}(n)}{n^s}
\]
for $\Re s$ sufficiently large.

If $\pi=\phi_*\sigma$ for some automorphic representation $\sigma$ of a smaller group, then we expect to have
\[
\ord_{s=1}L^S(s,\phi_*\sigma,\rho)=\ord_{s=1}L^S(s,\sigma,\rho\circ\phi)\geq m(\triv,\rho\circ\phi)
\]
and the equality holds when $\H$ is primitive, where $m(\triv,\rho\circ\phi)$ denotes the multiplicity of the trivial representation in $\rho\circ\phi$, which is a complex representation of $^L\H$.

Unfortunately, we know little about the analytic continuation and the order of the pole at $s=1$. However, by using the trace formula, we are able to compute the average 
\begin{equation}\label{eq:beyondendoscopy}
\sum_{\pi}m_{\pi}\ord_{s=1}L^{S}(s,\pi,\rho)\prod_{v\in S}\Tr(\pi_v(f_v))
\end{equation}
in a reasonable way, which is Langlands' \emph{Beyond Endoscopy} proposal \cite{langlands2004}.
Here, $\pi$ runs over all cuspidal representations and $f_v$ are nice functions on $\G(\QQ_v)$ for each $v\in S$. $m_\pi$ denotes the multiplicity of $\pi$ in $L^2_{\mathrm{cusp}}(\G(\QQ)\bs \G(\AA)^1)$.
We have
\[
\sum_{\pi}m_\pi a_{\pi,\rho}(p)\prod_{v\in S}\Tr(\pi_v(f_v))=I_{\mathrm{cusp}}(f^{p,\rho}),
\]
where $f^{p,\rho}=\bigotimes_{v\in S}f_v\otimes\bigotimes_{v\notin S}'f_v^{p,\rho}$, with $f_v^{p,\rho}$ spherical for $v\notin S$, and $\Tr(\pi_p(f_p^{p,\rho}))=a_{\pi,\rho}(p)$.

If $L^S(s,\pi,\rho)$ admits meromorphic continuation on $\Re s>1-\delta$ with neither zeros nor poles on the vertical line $s=1+\rmi t$ except for at $s=1$, then by the Ikehara theorem, we expect the following asymptotic formula:
\[
\lim_{X\to +\infty}\frac{1}{X}\sum_{\substack{p<X\\ p\notin S}}a_{\pi,\rho}(p)\log p=\ord_{s=1}L^S(s,\pi,\rho),
\]
where $p$ runs over all primes less than $X$.

If such asymptotic formula holds, \eqref{eq:beyondendoscopy} can be rewritten as
\[
\lim_{X\to +\infty}\frac{1}{X}\sum_{\pi}m_\pi\prod_{v\in S}\Tr(\pi_v(f_v))\sum_{\substack{p<X\\ p \notin S}}a_{\pi,\rho}(p)\log p .
\] 
and thus we expect that 
\[
\lim_{X\to +\infty}\frac{1}{X}\sum_{\substack{p<X\\ p \notin S}}\log pI_{\mathrm{cusp}}(f^{p,\rho})=\sum_{\pi}m_\pi \prod_{v\in S}\Tr(\pi_v(f_v))\ord_{s=1}L^S(s,\pi,\rho).
\]

We can consider a more general setting, which was proposed by Sarnak \cite{sarnak2001}. For $\gcd(n,S)=1$, we have
\[
\sum_{\pi}m_\pi a_{\pi,\rho}(n)\prod_{v\in S}\Tr(\pi_v(f_v))=I_{\mathrm{cusp}}(f^{n,\rho}),
\]
where $f^{n,\rho}=\bigotimes_{v\in S}f_v\otimes\bigotimes_{v\notin S}'f_v^{n,\rho}$, with $f_v^{n,\rho}$ spherical for $v\notin S$, and $\Tr(\pi_p(f_p^{n,\rho}))=a_{\pi,\rho}(p^{v_p(n)})$ for all $p\notin S$. By the Ikehara theorem, we expect to have
\[
\lim_{X\to +\infty}\frac{1}{X}\sum_{\substack{n<X\\ \gcd(n,S)=1}}a_{\pi,\rho}(n)=\res_{s=1}L^S(s,\pi,\rho),
\]
and we expect that
\[
\lim_{X\to +\infty}\frac{1}{X}\sum_{\substack{n<X\\ \gcd(n,S)=1}}I_{\mathrm{cusp}}(f^{n,\rho})=\sum_{\pi}m_\pi \prod_{v\in S}\Tr(\pi_v(f_v))\res_{s=1}L^S(s,\pi,\rho),
\]
where $\pi$ runs over certain representations of $\G(\AA)$ depending on $f^{n,\rho}$. For example, if $\G=\GL_2$ and $\rho$ is the standard representation, then $\res_{s=1}L^S(s,\pi,\rho)=0$ for all cuspidal $\pi$. Hence the expectation is
\begin{equation}\label{eq:standardrepresentation}
\lim_{X\to +\infty}\frac{1}{X}\sum_{\substack{n<X\\ \gcd(n,S)=1}}I_{\mathrm{cusp}}(f^{n,\rho})=0.
\end{equation}

For further background on Beyond Endoscopy and related problems, see \cite{langlands2004,langlands2010,arthur2017,espinosa2022,altug2024}.

We now assume that $\G=\GL_2$ and let $\rho$ be the standard representation.
For any prime number $p$ and $m\in \ZZ_{\geq 0}$, we define
\[
\cX_p^{m}=\{X\in \M_2(\ZZ_p)\,|\, \mathopen{|}\det X\mathclose{|}_p = p^{-m}\}.
\]
For example, if $m=0$, $\cX_p^{m}$ is just $\cK_p=\GL_2(\ZZ_p)$. By Hecke operator theory, we can choose $f^{n}$ to be $\bigotimes_{v\in \mf{S}}f^{n}_v$, where $\mf{S}$ denotes the set of places of $\QQ$, and
\begin{enumerate}[itemsep=0pt,parsep=0pt,topsep=2pt,leftmargin=0pt,labelsep=3pt,itemindent=9pt,label=\textbullet]
  \item If $v=p$, $f^{n}_p=p^{-n_p/2}\triv_{\cX_p^{n_p}}$, where $n_p=v_p(n)$.
  \item If $v=\infty$,  $f^{n}_\infty=f_\infty\in C^\infty(Z_+\bs \G(\RR))$ such that the orbital integrals are compactly supported modulo $Z_+$, and other than this condition they are arbitrary.
\end{enumerate}
In this case, $\pi$ runs over all unramified cuspidal representations. Venkatesh \cite{venkatesh2004} established an asymptotic formula for the residue case for $k\leq 2$, using the Petersson-Kuznetsov trace formula. As a new approach, Altu\u{g} \cite{altug2015,altug2017,altug2020} proved \eqref{eq:standardrepresentation} by using the Arthur-Selberg trace formula for $f_{\infty,m}$ that is a matrix coefficient of a weight $m$ discrete series for $m>2$ even. 
Altu\u{g} actually proved that
\[
\sum_{n<X}\Tr(T_m(n))\ll_{m,\varepsilon} X^{\frac{31}{32}+\varepsilon},
\]
where $T_m(n)$ denotes the (normalized) $n^{\mathrm{th}}$ Hecke operator acting on $S_m(\SL_2(\ZZ))$, the space of holomorphic cusp forms of weight $m$ for the modular group $\Gamma=\SL_2(\ZZ)$.

\subsection{Main results in this paper and proof strategy}
In this paper, building on the previous four papers \cite{cheng2025,cheng2025b,cheng2025c,cheng2026}, we prove \eqref{eq:standardrepresentation} directly over $\QQ$ in the ramified setting, with $\rho=\mathrm{Std}$ the standard representation.

We consider $S=\{\infty,q_1,\dots,q_r\}$ for primes $q_1,\dots,q_r$ such that $2\in S$, and a corresponding function $f^{n}=\bigotimes_{v\in \mf{S}}'f^n_v$ such that the local components at places in $S$ are arbitrary. 
Specifically, for $\gcd(n,S)=1$, the function $f^n$ is defined as follows:
\begin{enumerate}[itemsep=0pt,parsep=0pt,topsep=2pt,leftmargin=0pt,labelsep=3pt,itemindent=9pt,label=\textbullet]
  \item If $v=p\notin S$, we choose $f^n_v=p^{-n_p/2}\triv_{\cX_p^{n_p}}$, where $n_p=v_p(n)$.
  \item If $v=q_i\in S$ and is a prime, we choose $f^n_v=f_{q_i}\in C_c^\infty(\G(\QQ_{q_i}))$.
  \item If $v=\infty$, we choose $f^n_v=f_\infty\in C_c^\infty(Z_+\bs \G(\RR))$. 
\end{enumerate}

In this case, $f_v^{n}$ is spherical for $v\notin S$, and $\Tr(\pi_p(f_p^{n}))=a_{\pi}(p^{n_p})$ for all $p\notin S$.

To prove \eqref{eq:standardrepresentation}, it suffices to give an asymptotic formula for
\[
\sum_{\substack{n<X\\ \gcd(n,S)=1}}I_\cusp(f^n)
\]
and show that it is $o(X)$. Using Arthur-Selberg trace formula we may split $I_\cusp(f^n)$ into various terms. In \cite{cheng2026} we have done that completely and found equivalent conditions for \eqref{eq:standardrepresentation}. We recall the result now.
\begin{theorem}\label{thm:globalresult}
The following assertions are equivalent:
\begin{enumerate}[itemsep=0pt,parsep=0pt,topsep=0pt,leftmargin=0pt,labelsep=2.5pt, itemindent=15pt,label=\upshape{(\alph*)}]
  \item \eqref{eq:standardrepresentation} holds.
  \item Theorem 6.1 of \cite{cheng2026} holds.
  \item Theorem B.2, Theorem B.9 and Theorem B.10  of \cite{cheng2026} hold. Namely \eqref{eq:archimedeanresult} for the archimedean place, \eqref{eq:nonarchimedeanresult1} for $p\neq 2$, and \eqref{eq:nonarchimedeanresult2} for $p=2$ hold.
\end{enumerate}
\end{theorem}

Actually, \eqref{eq:standardrepresentation} has been proved (cf. \cite[Theorem 11.7.1]{getz2024} for example). Hence we know that (b) and (c) can be considered as proven results. However, the aim of beyond endoscopy is to give a new proof of \eqref{eq:standardrepresentation}. So we need to prove (b) or (c) \emph{directly}.

The proof is sketched as follows: First we use Arthur's result \autoref{thm:fourierweighted} on the Fourier transform of weighted orbital integrals. Specifically, we compute the functions $I_{\L,\M}^{\S}(\tau,\gamma)$ directly for $\tau=\triv$ assuming (c) of \autoref{thm:globalresult}, in the archimedean and the nonarchimedean cases, which involves complicated computations which can be reversed. Next, we compare the resulting formulas with the known results for Arthur, Herb and Sally \cite{arthur1985} and Hoffmann \cite{hoffmann1997,hoffmann2008} in the archimedean case, and perform a direct computation in \autoref{sec:explicitnonarchimedean} using Arthur's definition in the nonarchimedean case to match the results of the kernel function.

In sum, together with the result in this paper and \cite{cheng2025,cheng2025b,cheng2025c,cheng2026}, we proved that
\[
\lim_{X\to +\infty}\frac{1}{X}\sum_{\substack{n<X\\\gcd(n,S)=1}}I_\cusp(f^n)=0
\]
\emph{without any additional assumptions}. 

\subsection{Notations}
\begin{enumerate}[itemsep=0pt,parsep=0pt,topsep=0pt,leftmargin=0pt,labelsep=3pt,itemindent=9pt,label=\textbullet]
  \item $\# X$ denotes the number of elements in a set $X$.
  \item For $A\subseteq X$, $\triv_A$ denotes the characteristic function on $X$, defined by $\triv_A(x)=1$ for $x\in A$ and $\triv_A(x)=0$ for $x\notin A$.
  \item $\triv$ also denotes the trivial character or the trivial representation.
  \item We often use the notation $a\equiv b\,(n)$ to denote $a\equiv b\pmod n$.
  \item If $R$ is a ring (which we \emph{always} assume to be commutative with $1$), $R^\times$ denotes its group of units.
  \item $\SS^1$ denotes the group of complex numbers with absolute value $1$. 
\end{enumerate}
\section{Preliminaries}\label{sec:preliminaries}
In this section, we recall relevant definitions and results from previous work \cite{cheng2025,cheng2025b,cheng2025c,cheng2026}. To avoid notational conflicts between \cite{cheng2025b} and \cite{cheng2025c}, we adopt the conventions of the latter. Readers familiar with these papers may skip this section.

\subsection{The modified hyperbolic terms}
The hyperbolic part of the trace formula is given by
\[ \index{jhyp@$J_\hyp(f)$}
J_\hyp(f)=-\frac{1}{2}\sum_{\gamma\in \A(\QQ)_{\reg}}\int_{\A(\AA)\bs \G(\AA)}f(g^{-1}\gamma g)\alpha(H_\B(wg)+H_\B(g))\rmd g,
\]
where $\A$ is the diagonal torus, $w$ is the nontrivial element in the Weyl group of $(\G,\A)$, $\alpha$ denotes the positive root in $\mathfrak{sl}_2$, $\B$ denotes the subgroup of upper triangular matrices, and $H_\B$ denotes the Harish-Chandra map. It can be split into local weighted orbital integrals for $f=\bigotimes_{v\in \mf{S}}'f_v$. We have
\begin{align*}
&\int_{\A(\AA)\bs \G(\AA)}f(g^{-1}\gamma g)\alpha(H_\B(wg)+H_\B(g))\rmd g \\
=&\sum_{v\in \mf{S}}\int_{\A(\QQ_v)\bs \G(\QQ_v)}f_v(g_v^{-1}\gamma g_v)\alpha(H_\B(wg_v)+H_\B(g_v))\rmd g_v\cdot\prod_{w\neq v} \int_{\A(\QQ_w)\bs \G(\QQ_w)}f_w(g_w^{-1}\gamma g_w)\rmd g_w.
\end{align*}
Hence we may write
\[
J_\hyp(f)=\sum_{v\in \mf{S}}J_{\hyp,v}(f),
\]
where the \emph{local hyperbolic part} $J_{\hyp,v}(f)$ is defined to be
\[\index{jhyplocal@$J_{\hyp,v}(f)$}
-\frac{1}{2}\sum_{\gamma\in \A(\QQ)_{\reg}}\int_{\A(\QQ_v)\bs \G(\QQ_v)}f_v(g_v^{-1}\gamma g_v)\alpha(H_\B(wg_v)+H_\B(g_v))\rmd g_v\cdot\prod_{w\neq v} \int_{\A(\QQ_w)\bs \G(\QQ_w)}f_w(g_w^{-1}\gamma g_w)\rmd g_w.
\]
Also, we define the \emph{local weighted orbital integral} as
\begin{equation}\label{eq:modifiedlocalweightedorbital}\index{worb@$\worb(f;\gamma)$}
\worb(f_v;\gamma)=\int_{\A(\QQ_v)\bs \G(\QQ_v)}f_v(g_v^{-1}\gamma g_v)\alpha(H_\B(wg_v)+H_\B(g_v))\rmd g_v
\end{equation}
for $\gamma=(\begin{smallmatrix}\gamma_1 &  \\   & \gamma_2 \end{smallmatrix})\in \A(\QQ_v)_\reg$. The test function $f_v$ satisfies $f_v\in C_c^\infty(\G(\QQ_v))$ if $v$ is nonarchimedean, and $f_v\in C_c^\infty(Z_+\bs\G(\RR))$ if $v=\infty$.

The \emph{modified local weighted orbital integral (of the first kind)} is defined by \index{worb1@$\worb\sptilde(f;\gamma)$}
\begin{align*}
  \worb\sptilde(f;t) & =\int_{\A(\QQ_v)\bs \G(\QQ_v)}f(g^{-1}t g)\left(\alpha(H_\B(wg)+H_\B(g))-2\log\frac{|a-b|_v}{|ab|_v^{1/2}}\right)\rmd g \\
   & =\worb(f;t)-2\log\frac{|a-b|_v}{|ab|_v^{1/2}}\orb(f;t)
\end{align*}
for $t=(\begin{smallmatrix} a & 0 \\  0 & b \end{smallmatrix})$. It is $Z_+$-invariant with respect to $t$ when $v=\infty$, where \index{zplus@$Z_+$}$Z_+$ denotes the connected component of the identity matrix in the center $\Z(\RR)$ of $\G(\RR)$.

Finally, we define
\[ \index{tr1@$\widetilde{\Tr}$}
\widetilde{\Tr} \left(\Ind_{\B(\QQ_v)}^{\G(\QQ_v)}(s,\mu_v)(f_v)\right):=\int_{ \A(\QQ_v)}\frac{|x-y|_v}{|xy|_v^{1/2}}\left|\frac xy\right|_v^{s}\mu_{1,v}(x)\mu_{2,v}(y)\worb\sptilde(f_v;t)\rmd t
\]
for nonarchimedean $v$ and
\[  \index{tr1@$\widetilde{\Tr}$}
\widetilde{\Tr} \left(\Ind_{\B(\RR)}^{\G(\RR)}(s,\mu_\infty)(f_\infty)\right):=\int_{Z_+\bs\A(\RR)}\frac{|a-b|_\infty} {|ab|_\infty^{1/2}}\left|\frac ab\right|_\infty^{s}\mu_{1,\infty}(a)\mu_{2,\infty}(b)\worb\sptilde(f_\infty;t)\rmd t.
\]

\subsection{The modified $p$-adic norm}\label{subsec:modifiednorm}
For any prime $p$, the \emph{modified norm} $|\cdot|_p'$ \index{1ynorm@$\vert \cdot\vert_p'$} is defined as follows: 

For $p\neq 2$ and $y\in \QQ_p$, we define $|y|_p' = p^{-2\lfloor v_p(y)/2\rfloor}$.  
This satisfies $|y|_p' = |y|_p$ if $v_p(y)$ is even, and $|y|_p' = p |y|_p$ if $v_p(y)$ is odd.

For $p=2$ and $y\in \QQ_p$, we define
\[
|y|_p' = \begin{cases}
  p^{-v_p(y)}=|y|_p & \text{if $v_p(y)$ is even, $y_0\equiv 1\,(4)$}, \\
  p^{-v_p(y)+2}=p^2|y|_p & \text{if $v_p(y)$ is even, $y_0\equiv 3\,(4)$}, \\
  p^{-v_p(y)+3}=p^3|y|_p & \text{if $v_p(y)$ is odd},
\end{cases}
\]
where $y_0=yp^{-v_p(y)}$.
Clearly $|a^2y|_p'=|a|_p^2|y|_p'$ for any $a\in \QQ_p$.

Also, for any regular element $\gamma\in\G(\QQ_p)$, we denote $T_\gamma=\Tr\gamma$ and $N_\gamma=\det\gamma$. $k_\gamma$ is defined such that $p^{k_\gamma}=|T_\gamma^2-4N_\gamma|_p'^{-1/2}$. This coincides with the original definition of \cite{cheng2025} (see Proposition 2.6 of loc. cit.).

\subsection{Singularities of the orbital integrals}\label{subsec:singularities}
We follow the notations in \cite{cheng2025c}. Define
\[ \index{omegax@$\omega_\infty(x)$}
\omega_\infty(x)=\begin{cases}
             0, & x>0, \\
             1, & x<0
           \end{cases}
\]
for $x\in \RR$ with $x\neq 0$ and
\[ \index{omegay@$\omega_p(y)$}
\omega_p(y)=\legendresymbol{y|y|_{p}'}{p}
\]
for $p\in S$ and $y\in \QQ_p$ with $y\neq 0$. When $p=q_i$, we often write $\omega_p=\omega_i$\index{omegayi@$\omega_i(y_i)$}. For $\iota\in \{0,1\}$ we define
\[
X_{\iota}=\{x\in \RR\ |\ \omega_\infty(x)=\iota\} \index{xiota@$X_{\iota}$}
\]
and for $\epsilon_i\in \{0,\pm 1\}$, we define
\[
Y_{\epsilon_i}=\{y_i\in \QQ_{q_i}\ |\ \omega_i(y_i)=\epsilon_i\}. \index{yepsilon@$Y_{\epsilon}$}
\]

For any prime $p\in S$, we define \index{thetap@$\theta_{p}(\gamma)$}
\[
\theta_{p}(\gamma)=\frac{1}{\mathopen{|}\det\gamma\mathclose{|}_p^{1/2}}\left(1-\frac{\chi(p)}{p}\right)^{-1}p^{-{k_\gamma}}\orb(f_{p};\gamma),
\]
where $\chi(p)=\omega_p(\Tr\gamma,\det\gamma)$. By Corollary 2.12 of \cite{cheng2025}, the local behavior of $\theta_{p}$ at $z=aI$ is
\begin{equation}\label{eq:shalikalocal}
\theta_{p}(\gamma)=\lambda_1\left(1-\frac{\chi(p)}{p}\right)^{-1}p^{-{k_\gamma}} \frac{1-\chi(p)}{1-p}+\lambda_2.
\end{equation}

Clearly $\theta_{p}(\gamma)$ is invariant under conjugation. Thus $\theta_{p}(\gamma)$ can be parametrized by $T=\Tr\gamma$ and $N=\det\gamma$, i.e., $\theta_{p}(\gamma)=\theta_p(T,N)$\index{thetaptn@$\theta_{p}(T,N)$}. 
Since $\theta_p(\gamma)$ is smooth away from the center, $\theta_p(T,N)$ is smooth except at $T^2=4N$.

For the archimedean orbital integral, we recall the following theorem (cf. \cite[Theorem 2.13]{cheng2025}).
\begin{theorem}\label{thm:archimedeanintegral}
For any $f_\infty\in C^\infty(\G(\RR))$, any maximal torus $\T(\RR)$ in $\G(\RR)$ and any $z$ in the center of $\G(\RR)$, there exists a neighborhood $N$ in $\T(\RR)$ of $z$ and smooth functions $g_1,g_2\in C^\infty(N)$ (depending on $f_\infty$ and $z$) such that
\begin{equation}\label{eq:archimedeanintegral}
\orb(f_\infty;\gamma)=g_1(\gamma)+\frac{|\gamma_1\gamma_2|^{1/2}}{|\gamma_1-\gamma_2|}g_2(\gamma)
\end{equation}
for any $\gamma\in \T(\RR)$, where $\gamma_1$ and $\gamma_2$ are the eigenvalues of $\gamma$.  Moreover, $g_1$ and $g_2$ can be extended smoothly to all split and elliptic elements, remaining invariant under conjugation, with $g_1(\gamma)=0$ if $\T(\RR)$ is split, and $g_2$ can further be extended smoothly to the center. If $f_\infty$ is $Z_+$-invariant, then $g_1$ and $g_2$ are also $Z_+$-invariant.
\end{theorem}

We define \index{thetainfty@$\theta_{\infty}(\gamma)$}
\[
\theta_\infty(\gamma)=\frac{|\gamma_1-\gamma_2|}{|\gamma_1\gamma_2|^{1/2}}\orb(f_\infty;\gamma)= \frac{|\gamma_1-\gamma_2|}{|\gamma_1\gamma_2|^{1/2}}g_1(\gamma)+g_2(\gamma).
\]
Since $g_1,g_2$ and $\theta_\infty$ are invariant under conjugation, we can parametrize them by $T_\gamma$ and $N_\gamma$ as in the nonarchimedean case, i.e. \index{thetainftytn@$\theta_{\infty}(T,N)$}$\theta_{\infty}(\gamma)=\theta_\infty(T,N)$.

Since $T_{z\gamma}=aT_\gamma$ and  $N_{z\gamma}=a^2N_\gamma$ for $z=aI$ with $a>0$, we have $g_i(T_\gamma,N_\gamma)=g_i(aT_\gamma,a^2N_\gamma)$ and $\theta_\infty(T_\gamma,N_\gamma)=\theta_\infty(aT_\gamma,a^2N_\gamma)$ for $i=1,2$ and any $a>0$. 

Also, we set $\theta_\infty^\pm(x)=\theta_\infty(x,\pm 1/4)$\index{thetainftyx@$\theta_\infty^\pm(x)$}, which coincides with the notation in \cite{cheng2025} and \cite{cheng2025b}.

In \cite{cheng2025c} we have the following definitions. Let
\[ 
\Theta_\infty^\pm(x)=\theta_\infty\left(\pm 1,\frac{1-x}{4}\right). \index{thetabiginftypmx@$\Theta_\infty^\pm(x)$}
\]
and write \index{thetahatinftyx@$\widehat{\Theta}_\infty(x)$}$\widehat{\Theta}_\infty(x)=\Theta_\infty^+(x)+\Theta_\infty^-(x)$. Also, for any prime $p$, we define
\[ \index{thetahatpy@$\widehat{\Theta}_p(y)$}
\widehat{\Theta}_p(y)=\int_{\QQ_p^\times}\theta_p\left(z,\frac{z^2(1-y)}{4}\right)\frac{\rmd z}{|z|_p}.
\]

\section{The Fourier transform of weighted orbital integrals}
Let $f_p\in C_c^\infty(\G(\QQ_p))$ for primes $p$ and $f_\infty\in C_c^\infty(Z_+\bs \G(\RR))$, where $Z_+$ is the identity component of the center of $\G(\RR)$. Recall that what we want to prove are the following identities (cf. Theorem B.2, Theorem B.9, and Theorem B.10 of \cite{cheng2026}):

For the archimedean case,
\begin{equation}\label{eq:archimedeanresult}
\begin{split}
    &\frac12\Tr\left(R_\infty(0,\triv)^{-1}R_\infty'(0,\triv)\xi_0(f_\infty)\right)=\log 2\Tr(\xi_0(f_\infty))+ {\uppi\int_{X_1}|1-x|^{-1}\widehat{\Theta}_\infty(x)|x| ^{-\frac12}\rmd x} \\
   -& 2\int_{X_0}\log\frac{|1-x|}{|x|}|1-x|^{-1}\widehat{\Theta}_\infty(x)|x|^{-\frac12}\rmd x
   +\frac12\widetilde{\Tr}\left(\Ind_{\B(\RR)}^{\G(\RR)}(0,\triv)(f_\infty)\right),
\end{split}
\end{equation}
and for the nonarchimedean case, 
\begin{equation}\label{eq:nonarchimedeanresult1}
\begin{split}
    &\frac12\Tr\left(R_p(0,\triv)^{-1}R_p'(0,\triv)\xi_0(f_p)\right)=-\frac12(1+p^{-1})\frac{\log p}{1-p^{-1}}\Tr(\xi_0(f_p))\\
    +&\sum_{\epsilon\in \{0,-1\}}\frac{2(1-\epsilon p^{-1})\log p}{(1-\epsilon)(1-p^{-1})^2}\int_{Y_{\epsilon}}|1-y|_{p}^{-1}\widehat{\Theta}_{p}(y) |y|_{p}'^{-\frac12}\rmd y\\
   -&\frac{2}{1-p^{-1}}\int_{Y_{1}}\log\frac{|1-y|_{p}}{|y|_{p}'}|1-y|_{p}^{-1}\widehat{\Theta}_{p}(y) |y|_{p}'^{-\frac12}\rmd y
   +\frac12\widetilde{\Tr}\left(\Ind_{\B(\QQ_p)}^{\G(\QQ_p)}(0,\triv)(f_p)\right)
\end{split}
\end{equation}
for $p\neq 2$ and
\begin{equation}\label{eq:nonarchimedeanresult2}
\begin{split}
    &\frac12\Tr\left(R_p(0,\triv)^{-1}R_p'(0,\triv)\xi_0(f_p)\right)=-2\log 2\Tr(\xi_0(f_p))-\frac12(1+p^{-1})\frac{\log p}{1-p^{-1}}\Tr(\xi_0(f_p))\\
    +&\sum_{\epsilon\in \{0,-1\}}\frac{2(1-\epsilon p^{-1})\log p}{(1-\epsilon)(1-p^{-1})^2}\int_{Y_{\epsilon}}|1-y|_{p}^{-1}\widehat{\Theta}_{p}(y) |y|_{p}'^{-\frac12}\rmd y\\
   -&\frac{2}{1-p^{-1}}\int_{Y_{1}}\log\frac{|1-y|_{p}}{|y|_{p}'}|1-y|_{p}^{-1}\widehat{\Theta}_{p}(y) |y|_{p}'^{-\frac12}\rmd y
   +\frac12\widetilde{\Tr}\left(\Ind_{\B(\QQ_p)}^{\G(\QQ_p)}(0,\triv)(f_p)\right)
\end{split}
\end{equation}
for $p=2$.

To prove them, we use Arthur's result \cite{arthur1994} on the Fourier transform of weighted orbital integrals. The general version is as follows:

\begin{theorem}[\cite{arthur1994}, Corollary 4.4 and Remark 1]\label{thm:fourierweighted}
Let $F$ be a local field. Then there exist uniquely determined smooth functions \index{ilms@$I_{\L,\M}^\S(\tau,\gamma)$} $I_{\L,\M}^\S(\tau,\gamma)$ such that
\begin{equation}\label{eq:fouriertransformweighted44}
\index{iltaug@$I_\L(\tau,g)$|seeonly{\cite{arthur1994}}}\index{jltaug@$J_\L(\tau,g)$|seealsopage {\cite{arthur1994}}}\index{gammaell@$\Gamma_\el$|seealsopage {\cite{arthur1994}}}J_\L(\tau,g)=\sum_{\M\in\mathscr{L}}\int_{\Gamma_\el(\M(F))}\sum_{\S\in \mathscr{L}(\L)\cap\mathscr{L}(\M)}\#W_0^\M(\#W_0^\S)^{-1}I_{\L,\M}^\S(\tau,\gamma)J_\S(\gamma,g)\rmd \gamma.
\end{equation}
For the notation we refer to the first four sections of \cite{arthur1994}.
\end{theorem}
\begin{remark}
The function $I_{\L,\M}^\S(\tau,\gamma)$ was written as $I_{\L}^\S(\tau,\gamma)$ in \cite{arthur1994}. We emphasize this $\M$ since we will consider two functions with the same $\L$ and $\S$, but actually they are different.
\end{remark}

For our purposes, we only need the following identity
\begin{equation}\label{eq:fouriertransformweighted2}
\begin{split}
    J_\A(\triv,g)& =\int_{\Gamma_\el(\G(F))}I_{\A,\G}^{\G}(\triv,\gamma)J_\G(\gamma,g)\rmd \gamma+\int_{\Gamma_\el(\A(F))}I_{\A,\A}^{\A}(\triv,\gamma)J_\A(\gamma,g)\rmd \gamma \\
     & +\frac12\int_{\Gamma_\el(\A(F))}I_{\A,\A}^{\G}(\triv,\gamma)J_\G(\gamma,g)\rmd \gamma,
\end{split}
\end{equation}
where integration over the elliptic conjugacy classes $\Gamma_\el(\G(F))$ is given by
\[
\int_{\Gamma_\el(\G(F))}\phi(\gamma)\rmd \gamma=\sum_{\T}\frac{1}{\#W(\G,\T)}\int_{\T(F)}\phi(\gamma)\rmd \gamma,
\]
where $\T$ runs over all elliptic maximal tori in $\G$ over $F$. In particular,
\[
\int_{\Gamma_\el(\A(F))}\phi(\gamma)\rmd \gamma=\int_{\A(F)}\phi(\gamma)\rmd \gamma.
\]

In our language, \eqref{eq:fouriertransformweighted2} becomes the following theorem.
\begin{theorem}\label{thm:normalizedintertwining}
Let $v$ be a place of $\QQ$. Let $D(\gamma)$ denote the discriminant of $\gamma$. Let $f_v\in C_c^\infty(\G(\QQ_v))$ if $v$ is nonarchimedean and $f_v\in C_c^\infty(Z_+\bs \G(\RR))$ if $v=\infty$. Then we have
\begin{align*}
  \Tr\left(R_v(0,\triv)^{-1}R_v'(0,\triv)\xi_0(f)\right) & =\int_{\Gamma_\el(\G(\QQ_v))}I_{\A,\G}^\G(\triv,\gamma)|D(\gamma)|^{1/2}\orb(f;\gamma)\rmd \gamma \\
 +\int_{\Gamma_\el(\A(\QQ_v))}I_{\A,\A}^\A(\triv,\gamma)|D(\gamma)|^{1/2}\worb(f;\gamma)\rmd \gamma  & +\frac12\int_{\Gamma_\el(\A(\QQ_v))}I_{\A,\A}^\G(\triv,\gamma)|D(\gamma)|^{1/2}\orb(f;\gamma)\rmd \gamma
\end{align*}
for certain functions $I_{\A,\G}^\G(\tau,\gamma)$, $I_{\A,\A}^\A(\tau,\gamma)$, and $I_{\A,\A}^\G(\tau,\gamma)$.
\end{theorem}

To compute these functions, we need the following reciprocity relation \cite[Theorem 4.5]{arthur1994}: We consider the general case. We have
\[ \index{ilmsdual@$I_{\L,\M}^\S(\gamma,\tau)$}
I_{\L,\M}^{\S}(\gamma,\tau)=(-1)^{\dim(A_\L\times A_\M)}i^\L(\tau)I_{\L,\M}^{\S}(\tau\spcheck,\gamma) \index{iltau@$i^\L(\tau)$|seealsopage{\cite{arthur1994}}},
\]
where \index{am@$A_\M$}$A_\M$ denotes the split component of the center of $\M$. $I_{\L,\M}^{\S}(\gamma,\tau)$ is uniquely determined such that (cf. \cite[Corollary 4.2 and Remark 1]{arthur1994})
\begin{equation}\label{eq:fouriertransformweighted42}
\index{imgammaf@$I_\M(\gamma,f)$|seeonly {\cite{arthur1994}}}\index{jmgammaf@$J_\M(\gamma,f)$|seealsopage {\cite{arthur1994}}}\index{tdisc@$T_{\mathrm{disc}}$|seealsopage {\cite{arthur1994}}}J_\M(\gamma,f)=\sum_{\L\in\mathscr{L}}\int_{T_\disc(\M(F))}\sum_{\S\in \mathscr{L}(\L)\cap\mathscr{L}(\M)}\#W_0^\L(\#W_0^\S)^{-1}I_{\L,\M}^\S(\gamma,\tau) J_\S(\tau,f)\rmd \tau.
\end{equation}
For other notations we refer \cite{arthur1994}. 

In our situation, we have $i^\A(\triv)=1$ and hence
\begin{equation}\label{eq:fouriertransformweightedreciprocity}
I_{\A,\G}^\G(\triv,\gamma)=-I_{\A,\G}^\G(\gamma,\triv),\quad I_{\A,\A}^\A(\triv,\gamma)=I_{\A,\A}^\A(\gamma,\triv),\quad\text{and}\quad I_{\A,\A}^\G(\triv,\gamma)=I_{\A,\A}^\G(\gamma,\triv).
\end{equation}
Moreover, by \eqref{eq:fouriertransformweighted42} we have
\begin{equation}\label{eq:fouriertransformweighteddual}
\begin{split}
    J_\A(\gamma,f)& =\int_{T_\disc(\G)}I_{\G,\A}^{\G}(\gamma,\tau)J_\G(\tau,f)\rmd \tau+\int_{T_\disc(\A)}I_{\A,\A}^{\A}(\gamma,\tau)J_\A(\tau,f)\rmd \tau \\
     & +\frac12\int_{T_\disc(\A)}I_{\A,\A}^{\G}(\gamma,\tau)J_\G(\tau,f)\rmd \tau
\end{split}
\end{equation}
and
\begin{equation}\label{eq:fouriertransformweighteddual2}
J_\G(\gamma,f)=\int_{T_\disc(\G)}I_{\G,\G}^{\G}(\gamma,\tau)J_\G(\tau,f)\rmd \tau+\frac12\int_{T_\disc(\A)}I_{\A,\G}^{\G}(\gamma,\tau)J_\G(\tau,f)\rmd \tau.
\end{equation}

In the next two sections, we will show that \eqref{eq:standardrepresentation} is equivalent to the explicit values of $I_{\A,\G}^\G(\triv,\gamma)$, $I_{\A,\A}^\A(\triv,\gamma)$ and $I_{\A,\A}^\G(\triv,\gamma)$ by using the results in Appendix B of \cite{cheng2026}.

\section{Archimedean theory}\label{sec:archimedean}
\noindent\textbf{Notations:} In this section we define $\G=\GL_2$, $\A$ the diagonal torus of $\G$, and $\Z$ the center of $\G$. We set $G=\G(\RR)$, $A=\A(\RR)$, $Z=\Z(\RR)$, and $Z_+$ the identity component of $Z$.

The main goal of this section is to prove the following theorem:
\begin{theorem}\label{thm:archimedean}
The following assertions are equivalent:
\begin{enumerate}[itemsep=0pt,parsep=0pt,topsep=0pt,leftmargin=0pt,labelsep=2.5pt, itemindent=15pt,label=\upshape{(\alph*)}]
  \item \eqref{eq:archimedeanresult} holds.
  \item We have 
  \[
  I_{\A,\G}^{\G}(\triv,\gamma)=2\uppi,\quad I_{\A,\A}^{\A}(\triv,\gamma)=1,\quad and \quad I_{\A,\A}^{\G}(\triv,\gamma)=-4\log 2+2\log |D(\gamma)|,
  \]
  where $D(\gamma)$ denotes the discriminant of $\gamma$.
  \item We have
  \[
  I_{\A,\G}^{\G}(\gamma,\triv)=-2\uppi,\quad I_{\A,\A}^{\A}(\gamma,\triv)=1,\quad and \quad I_{\A,\A}^{\G}(\gamma,\triv)=-4\log 2+2\log |D(\gamma)|.
  \]
\end{enumerate}
\end{theorem}

To prove this theorem we need to compute the archimedean integrals explicitly.

\begin{proposition}
For $f_\infty\in C_c^\infty(Z_+\bs G)$ we have
\[
\int_{-\infty}^{0}\frac{\widehat{\Theta}_\infty(x)}{|1-x||x|^{1/2}}\rmd x=\int_{0}^{2\uppi}\theta_\infty^+(\cos\varphi)\rmd\varphi.
\]
\end{proposition}
\begin{proof}
For $x<0$ we have 
\[
\Theta_\infty^\pm(x)=\theta_\infty\left(\pm 1,\frac{1-x}{4}\right)=\theta_\infty^+\left(\pm \frac{1}{\sqrt{1-x}}\right)
\]
by (3.3) of \cite{cheng2026}. Hence
\[
\int_{-\infty}^{0}\frac{\widehat{\Theta}_\infty(x)}{|1-x||x|^{1/2}}\rmd x=\int_{-\infty}^{0}\frac{1}{|1-x||x|^{1/2}}\theta_\infty^+\left( \frac{1}{\sqrt{1-x}}\right)\rmd x+\int_{-\infty}^{0}\frac{1}{|1-x||x|^{1/2}}\theta_\infty^+\left(- \frac{1}{\sqrt{1-x}}\right)\rmd x.
\]
For the first term, we make change of variable $x=-\tan^2\varphi$ so that $1/\sqrt{1-x}=\cos\varphi$ and $\rmd x=-2\tan\varphi\sec^2\varphi\,\rmd\varphi$. Hence
\begin{align*}
\int_{-\infty}^{0}\frac{1}{|1-x||x|^{1/2}}\theta_\infty^+\left( \frac{1}{\sqrt{1-x}}\right)\rmd x&=\int_{0}^{\uppi/2}\cos^2\varphi\,\theta_\infty^+(\cos\varphi)\cot\varphi\cdot 2\tan\varphi\sec^2\varphi\,\rmd \varphi\\
&=2\int_{0}^{\uppi/2}\theta_\infty^+(\cos\varphi)\rmd \varphi=\int_{-\uppi/2}^{\uppi/2}\theta_\infty^+(\cos\varphi)\rmd \varphi.
\end{align*}
Similarly,
\[
\int_{-\infty}^{0}\frac{1}{|1-x||x|^{1/2}}\theta_\infty^+\left(-\frac{1}{\sqrt{1-x}}\right)\rmd x= \int_{\uppi/2}^{3\uppi/2}\theta_\infty^+(\cos\varphi)\rmd \varphi.
\]
The conclusion now holds by adding the two integrals together and using periodicity.
\end{proof}

\begin{proposition}
For $f_\infty\in C_c^\infty(Z_+\bs G)$ we have
\[
\int_{\Gamma_\el(G)}|D(\gamma)|^{1/2}\orb(f_\infty;\gamma)\rmd\gamma=\int_{0}^{2\uppi}\theta_\infty^+(\cos\varphi)\rmd\varphi.
\]
\end{proposition}
\begin{proof}
By the definition of the measure on $\Gamma_\el(G)$ we have
\[
\int_{\Gamma_\el(G)}|D(\gamma)|^{1/2}\orb(f_\infty;\gamma)\rmd\gamma=\frac12\int_{Z_+\bs\CC^\times} |D(\gamma)|^{1/2}\orb(f_\infty;\gamma)\rmd \gamma=\frac12\int_{Z_+\bs\CC^\times} \theta_\infty(\gamma)\rmd\gamma.
\]
where $\rmd\gamma$ denotes the normalized measure \cite{arthur1993} on $Z_+\bs \CC^\times\cong \SS^1$, which is $2\rmd \theta$. 
Hence
\begin{align*}
\int_{\Gamma_\el(G)}|D(\gamma)|^{1/2}\orb(f_\infty;\gamma)\rmd\gamma&=\int_{0}^{2\uppi}\theta_\infty \begin{pmatrix}
\cos\varphi & \sin\varphi\\ -\sin\varphi & \cos\varphi
\end{pmatrix}\rmd \varphi\\
&=\int_{0}^{2\uppi}\theta_\infty(2\cos\varphi,1)\rmd\varphi =\int_{0}^{2\uppi}\theta_\infty^+(\cos\varphi)\rmd\varphi.\qedhere
\end{align*}
\end{proof}
\begin{proposition}\label{prop:tracenormal}
For $f_\infty\in C_c^\infty(Z_+\bs G)$ we have
\[
\Tr\left(\xi_0(f_\infty)\right)=\int_{Z_+\bs A}|D(\gamma)|^{1/2}\orb(f_\infty;\gamma)\rmd\gamma.
\]
\end{proposition}
\begin{proof}
For any $\gamma\in A$ we write $\gamma=(\begin{smallmatrix} a &\\ & b\end{smallmatrix})$. Then
\[
|D(\gamma)|=\frac{|a-b|^2}{|ab|}.
\]
Hence the result follows from the local version of \cite[Proposition 8.2]{cheng2025}. 
\end{proof}
\begin{proposition}\label{prop:tracetilde}
For $f_\infty\in C_c^\infty(Z_+\bs G)$ we have
\[
\widetilde{\Tr}\left(\Ind_B^G(0,\triv)(f_\infty)\right)=\int_{Z_+\bs A}|D(\gamma)|^{1/2}\worb(f_\infty;\gamma)\rmd\gamma-2\int_{Z_+\bs A}|D(\gamma)|^{1/2}\log|D(\gamma)|^{1/2}\orb(f_\infty;\gamma)\rmd \gamma.
\]
\end{proposition}
\begin{proof}
The result follows immediately from the definitions of $\widetilde{\Tr}$ and $\worb\sptilde$.
\end{proof}

\begin{proposition}
For $f_\infty\in C_c^\infty(Z_+\bs G)$ we have
\begin{align*}
\int_{0}^{+\infty}\log\frac{|1-x|}{|x|}\frac{\widehat{\Theta}_\infty(x)}{|x|^{1/2}|1-x|}\rmd x
=&\log2\int_{Z_+\bs A}|D(\gamma)|^{1/2}\orb(f_\infty;\gamma)\rmd\gamma\\
-&\int_{Z_+\bs A}|D(\gamma)|^{1/2}\log |D(\gamma)|^{1/2}\orb(f_\infty;\gamma)\rmd\gamma.
\end{align*}
\end{proposition}
\begin{proof}
By (3.3) of \cite{cheng2026} we have
\begin{align*}
&\int_{0}^{+\infty}\log\frac{|1-x|}{|x|}\frac{\widehat{\Theta}_\infty(x)}{|x|^{1/2}|1-x|}\rmd x \\
=&\int_{0}^{1}\log\frac{|1-x|}{|x|}\sum_\pm \frac{\theta^+_\infty(\pm\frac{1}{\sqrt{1-x}})}{|x|^{1/2}|1-x|}\rmd x+ \int_{1}^{+\infty}\log\frac{|1-x|}{|x|}\sum_\pm \frac{\theta^-_\infty(\pm\frac{1}{\sqrt{x-1}})}{|x|^{1/2}|1-x|}\rmd x.
\end{align*}
Now we consider the two terms above separately. For the first integral, we make the change of variable
\[
x=\frac{(t^{1/2}-t^{-1/2})^2}{(t^{1/2}+t^{-1/2})^2}
\]
so that
\[
\frac{1}{\sqrt{1-x}}=\frac{t^{1/2}+t^{-1/2}}{2}\quad\text{and}\quad \rmd x=\frac{4(t-1)}{(t+1)^3}\rmd t.
\]
Therefore
\begin{align*}
&\int_{0}^{1}\log\frac{|1-x|}{|x|}\sum_\pm \frac{\theta^+_\infty(\pm\frac{1}{\sqrt{1-x}})}{|x|^{1/2}|1-x|}\rmd x\\
=&\int_{1}^{+\infty}\log\frac{4}{(t^{1/2}-t^{-1/2})^2}\sum_{\pm}\theta_\infty^+\left(\pm \frac{t^{1/2}+t^{-1/2}}{2}\right)\left(\frac{4}{(t^{1/2}+t^{-1/2})^2} \frac{t^{1/2}-t^{-1/2}}{t^{1/2}+t^{-1/2}}\right)^{\!\!-1}4\frac{t^{1/2}-t^{-1/2}} {(t^{1/2}+t^{-1/2})^3}\frac{\rmd t}{t}\\
=&\int_{1}^{+\infty}\log\frac{4}{(t^{1/2}-t^{-1/2})^2}\sum_{\pm}\theta_\infty\begin{pmatrix}
                                                                              \pm t^{1/2} & 0 \\
                                                                              0 &\pm t^{-1/2} 
                                                                            \end{pmatrix}\frac{\rmd t}{t}\\
=&\frac12\int_{0}^{+\infty}\log\frac{4}{(t^{1/2}-t^{-1/2})^2}\sum_{\pm}\theta_\infty\begin{pmatrix}
                                                                              \pm t^{1/2} & 0 \\
                                                                              0 & \pm t^{-1/2} 
                                                                            \end{pmatrix}\frac{\rmd t}{t},
\end{align*}
where in the last step we used the symmetry. Now we consider the second integral. In this case, we use the changing of variable
\[
x=\frac{(t^{1/2}+t^{-1/2})^2}{(t^{1/2}-t^{-1/2})^2}
\]
so that
\[
\frac{1}{\sqrt{x-1}}=\frac{t^{1/2}-t^{-1/2}}{2}\quad\text{and}\quad \rmd x=-\frac{4(t+1)}{(t-1)^3}\rmd t.
\]
Therefore
\begin{align*}
&\int_{1}^{+\infty}\log\frac{|1-x|}{|x|}\sum_\pm \frac{\theta^-_\infty(\pm\frac{1}{\sqrt{x-1}})}{|x|^{1/2}|1-x|}\rmd x\\
=&\int_{1}^{+\infty}\log\frac{4}{(t^{1/2}+t^{-1/2})^2}\sum_{\pm}\theta_\infty^-\left(\pm \frac{t^{1/2}-t^{-1/2}}{2}\right)\left(\frac{4}{(t^{1/2}-t^{-1/2})^2} \frac{t^{1/2}+t^{-1/2}}{t^{1/2}-t^{-1/2}}\right)^{\!\!-1}4\frac{t^{1/2}+t^{-1/2}}{(t^{1/2}-t^{-1/2})^3} \frac{\rmd t}{t}\\
=&\int_{1}^{+\infty}\log\frac{4}{(t^{1/2}+t^{-1/2})^2}\sum_{\pm}\theta_\infty\begin{pmatrix}
                                                                              \pm t^{1/2} &0 \\
                                                                              0 &\mp t^{-1/2} 
                                                                            \end{pmatrix}\frac{\rmd t}{t}\\
=&\frac12\int_{0}^{+\infty}\log\frac{4}{(t^{1/2}+t^{-1/2})^2}\sum_{\pm}\theta_\infty\begin{pmatrix}
                                                                              \pm t^{1/2} & 0 \\
                                                                              0 & \mp t^{-1/2} 
                                                                            \end{pmatrix}\frac{\rmd t}{t}.
\end{align*}

In sum, we proved that
\begin{align*}
\int_{0}^{+\infty}\log\frac{|1-x|}{|x|}\frac{\widehat{\Theta}_\infty(x)}{|x|^{1/2}|1-x|}\rmd x
&=\int_{0}^{+\infty}\log\frac{2}{|t^{1/2}-t^{-1/2}|}\sum_{\pm}\theta_\infty\begin{pmatrix}
                                                                              \pm t^{1/2} & 0 \\
                                                                              0 & \pm t^{-1/2} 
                                                                            \end{pmatrix}\frac{\rmd t}{t}\\
&+\int_{0}^{+\infty}\log\frac{2}{|t^{1/2}+t^{-1/2}|}\sum_{\pm}\theta_\infty\begin{pmatrix}
                                                                              \pm t^{1/2} & 0 \\
                                                                              0 & \mp t^{-1/2}
                                                                            \end{pmatrix}\frac{\rmd t}{t}.
\end{align*}

Now we consider the right hand side of the formula in this proposition. Write $\gamma=(\begin{smallmatrix} a & \\ & b\end{smallmatrix})$. Then we have
\begin{align*}
\int_{Z_+\bs A}|D(\gamma)|^{1/2}\orb(f_\infty;\gamma)\rmd\gamma&=\int_{Z_+\bs (\RR^\times\times\RR^\times)}\theta_\infty\begin{pmatrix}
                                            a &  \\
                                             & b 
                                          \end{pmatrix}\rmd^\times a\rmd^\times b\\
&=\sum_{\epsilon_1,\epsilon_2\in \{\pm 1\}}\int_{Z_+\bs (\RR_{>0}\times\RR_{>0})}\theta_\infty\begin{pmatrix}
                                            \epsilon_1 \sqrt{a/b} &  \\
                                             & \epsilon_2 \sqrt{b/a} 
                                          \end{pmatrix}\rmd^\times a\rmd^\times b.
\end{align*}
and
\begin{align*}
&\int_{Z_+\bs A}|D(\gamma)|^{1/2}\log|D(\gamma)|^{1/2}\orb(f_\infty;\gamma)\rmd\gamma=\int_{Z_+\bs (\RR^\times\times\RR^\times)}\log\frac{|a-b|}{|ab|^{1/2}}\theta_\infty\begin{pmatrix}
                                            a &  \\
                                             & b 
                                          \end{pmatrix}\rmd^\times a\rmd^\times b\\
=&\sum_{\epsilon_1,\epsilon_2\in \{\pm 1\}}\int_{Z_+\bs (\RR_{>0}\times\RR_{>0})}\log\left|\epsilon_1\sqrt{\frac{a}{b}}-\epsilon_2 \sqrt{\frac{b}{a}}\right| \theta_\infty\begin{pmatrix}
                                            \epsilon_1 \sqrt{a/b} &  \\
                                             & \epsilon_2 \sqrt{b/a} 
                                          \end{pmatrix}\rmd^\times a\rmd^\times b.
\end{align*}
Using Lemma A.1 of \cite{cheng2026}, we know that
\[
\int_{Z_+\bs A}|D(\gamma)|^{1/2}\orb(f_\infty;\gamma)\rmd\gamma= \int_{0}^{+\infty}\sum_{\pm}\theta_\infty\begin{pmatrix}\pm t^{1/2} & 0 \\0 & \pm t^{-1/2} \end{pmatrix}\frac{\rmd t}{t}
+ \int_{0}^{+\infty}\sum_{\pm}\theta_\infty\begin{pmatrix}\pm t^{1/2} & 0 \\0 & \mp t^{-1/2} \end{pmatrix}\frac{\rmd t}{t}
\]
and
\begin{align*}
\int_{Z_+\bs A}|D(\gamma)|^{1/2}\log|D(\gamma)|^{1/2}\orb(f_\infty;\gamma)\rmd\gamma
&=\int_{0}^{+\infty}\log|t^{1/2}-t^{-1/2}|\sum_{\pm}\theta_\infty\begin{pmatrix}
                                                                              \pm t^{1/2} & 0 \\
                                                                              0 & \pm t^{-1/2} 
                                                                            \end{pmatrix}\frac{\rmd t}{t}\\
&+\int_{0}^{+\infty}\log|t^{1/2}+t^{-1/2}|\sum_{\pm}\theta_\infty\begin{pmatrix}
                                                                              \pm t^{1/2} & 0 \\
                                                                              0 & \mp t^{-1/2}
                                                                            \end{pmatrix}\frac{\rmd t}{t}.
\end{align*}
Hence we obtain the desired formula.
\end{proof}

Now we have all the ingredients for proving \autoref{thm:archimedean}.

\begin{proof}[Proof of \autoref{thm:archimedean}]
(b) and (c) are equivalent by \eqref{eq:fouriertransformweightedreciprocity}. Hence it suffices to show that (a) and (b) are equivalent. 

Assume that (a) holds, namely
\begin{align*}
    &\Tr\left(R_\infty(0,\triv)^{-1}R_\infty'(0,\triv)\xi_0(f_\infty)\right)=2\log 2\Tr(\xi_0(f_\infty))+ {2\uppi\int_{X_1}|1-x|^{-1}\widehat{\Theta}_\infty(x)|x| ^{-\frac12}\rmd x} \\
   -& 4\int_{X_0}\log\frac{|1-x|}{|x|}|1-x|^{-1}\widehat{\Theta}_\infty(x)|x|^{-\frac12}\rmd x
   +\widetilde{\Tr}\left(\Ind_{B}^{G}(0,\triv)(f_\infty)\right).
\end{align*}
Using \autoref{thm:normalizedintertwining} and substituting the results of this section into the above formula, we conclude that
\begin{align*}
&\int_{\Gamma_\el(G)}I_{\A,\G}^\G(\triv,\gamma)|D(\gamma)|^{1/2}\orb(f_\infty;\gamma)\rmd \gamma +\int_{\Gamma_\el(A)}I_{\A,\A}^\A(\triv,\gamma)|D(\gamma)|^{1/2}\worb(f_\infty;\gamma)\rmd \gamma  \\
+&\frac12\int_{\Gamma_\el(A)}I_{\A,\A}^\G(\triv,\gamma)|D(\gamma)|^{1/2}\orb(f_\infty;\gamma)\rmd \gamma \\ 
 =&2\log 2\int_{Z_+\bs A}|D(\gamma)|^{1/2}\orb(f_\infty;\gamma)\rmd \gamma+2\uppi \int_{\Gamma_\el(G)}|D(\gamma)|^{1/2}\orb(f_\infty;\gamma)\rmd\gamma\\
 -&4\log2\int_{Z_+\bs A}|D(\gamma)|^{1/2}\orb(f_\infty;\gamma)\rmd\gamma+4\int_{Z_+\bs A}|D(\gamma)|^{1/2}\log |D(\gamma)|^{1/2}\orb(f_\infty;\gamma)\rmd\gamma\\
   +&\int_{Z_+\bs A}|D(\gamma)|^{1/2}\worb(f_\infty;\gamma)\rmd\gamma
   -2\int_{Z_+\bs A}|D(\gamma)|^{1/2}\log|D(\gamma)|^{1/2}\orb(f_\infty;\gamma)\rmd \gamma.
\end{align*}
By comparing the coefficients and the uniqueness of the kernel functions we obtain (b). (b) also implies (a) by reversing the proof above.
\end{proof}

\section{Nonarchimedean theory}\label{sec:nonarchimedean}
Now we pass to the nonarchimedean case. We assume that $p$ is a fixed prime. 

\noindent\textbf{Notations:} In this section we define $\G$, $\A$, $\Z$ as in the previous section. We set $G=\G(\QQ_p)$, $A=\A(\QQ_p)$, and $Z=\Z(\QQ_p)$. $|\cdot|$ and $|\cdot|'$ are the norm and the modified norm on $\QQ_p$, respectively.

\begin{proposition}
Let $f_p\in C_c^\infty(G)$. Then we have
\[
\int_{\Gamma_\el(G)}|D(\gamma)|^{1/2}\orb(f_p;\gamma)\rmd \gamma=\frac{1}{1-p^{-1}} \sum_{\epsilon\in \{0,-1\}} \int_{Y_{\epsilon}}\frac{\widehat{\Theta}_p(y)}{|1-y||y|^{1/2}}\rmd y.
\]
More generally, for any elliptic torus $\T$ we have
\[
\frac12\int_{\T(\QQ_p)}|D(\gamma)|^{1/2} \orb(f_p;\gamma)\rmd \gamma=\frac{1}{1-p^{-1}} \int_{\gamma_y\in \T(\QQ_p)}\frac{\widehat{\Theta}_p(y)}{|1-y||y|^{1/2}}\rmd y,
\]
where $\gamma_y$ denotes any representative of elements in $G$ with discriminant $y$. 
\end{proposition}
\begin{proof}
By the definition of the measure on $\Gamma_\el(G)$, we have
\[
\int_{\Gamma_\el(G)}|D(\gamma)|^{1/2}\orb(f_p;\gamma)\rmd \gamma=\frac12\sum_{\T}\int_{\T(\QQ_p)}|D(\gamma)|^{1/2} \orb(f_p;\gamma)\rmd \gamma,
\]
where $\T$ runs over all elliptic maximal tori of $G$. Now for each torus we consider
\begin{equation}\label{eq:weylintegration}
\frac12\int_{\T(\QQ_p)}|D(\gamma)|^{1/2} \orb(f_p;\gamma)\rmd \gamma
\end{equation}

First we make the change of variable as in the proof of \cite[Proposition 7.3]{cheng2025} for elliptic tori. We recall the construction of this transformation: Note that $\T(\QQ_{p})\cong E_{p}^\times$ is an elliptic torus, where $E_{p}/\QQ_{p}$ is a quadratic extension.

For any $\gamma\in \T(\QQ_{p})$, recall that 
\[
|D(\gamma)|=\frac{|\gamma_1-\gamma_2|^2}{|\gamma_1\gamma_2|},
\]
where $\gamma_1,\gamma_2\in E_p$ are the eigenvalues of $\gamma$. 
Now we identify $\gamma$ with $\gamma_1$ so that $\gamma_2=\overline{\gamma}$ is the conjugate of $\gamma$. Then \eqref{eq:weylintegration} becomes
\[
\frac{1}{2}\int_{E_{p}^\times}\frac{|(\gamma-\overline{\gamma})^2|^{1/2}} {|\gamma\overline{\gamma}|^{1/2}} \orb(f_p;\gamma) \rmd\gamma,
\]
where $\rmd\gamma$ denotes the measure on $E_{p}^\times$.

Suppose that $\cO_{E_{p}}=\ZZ_{p}\oplus \ZZ_{p}\Delta$. Then we can identify $E_{p}$ with $\QQ_{p}^2$ via the basis $\{1,\Delta\}$. Thus we can embed $E_p$ into $\G(\QQ_{p})$ by identifying $\xi\in E_p$ with the matrix of the linear transformation $\eta\mapsto \xi\eta$ with respect to $\{1,\Delta\}$. If we write $\Delta^2=\alpha+\beta\Delta$, then 
\[
1=\begin{pmatrix}
  1 & 0 \\
  0 & 1 
\end{pmatrix}\qquad \text{and}\qquad \Delta=\begin{pmatrix}
  0 & \alpha \\
  1 & \beta 
\end{pmatrix}.
\]
Under the identification $\gamma \leftrightarrow a+b\Delta$, we have
\[
\rmd \gamma=\frac{\vol(\cO_{E_{p}})}{\vol(\cO_{E_{p}}^\times)}\frac{\rmd a\rmd b}{\mathopen{|}\det(a+b\Delta)\mathclose{|}},\quad|(\gamma-\overline{\gamma})^2|=|b|^2|(\Delta-\overline{\Delta})^2|,\quad \text{and}\quad |\gamma\overline{\gamma}|=\mathopen{|}\det(a+b\Delta)\mathclose{|}.
\]
Thus \eqref{eq:weylintegration} becomes
\[
\frac{1}{2}\int_{\QQ_{p}^2}\frac{|b||(\Delta-\overline{\Delta})^2|^{1/2}} {\mathopen{|}\det(a+b\Delta)\mathclose{|}^{1/2}}\orb(f_p;\gamma)\frac{\vol(\cO_{E_{p}})} {\vol(\cO_{E_{p}}^\times)}\frac{\rmd a\rmd b}{\mathopen{|}\det(a+b\Delta)\mathclose{|}}.
\]
Now we make the change of variable $T=\Tr(a+b\Delta)=2a+b\beta$ and $N=\det(a+b\Delta)=a^2+\beta ab-\alpha b^2$. Then we have
\[
b=\pm \sqrt{\frac{T^2-4N}{\beta^2+4\alpha}}\qquad\text{and}\qquad a=\frac{T-\beta b}{2}.
\]
As in \cite{cheng2025}, we take the plus sign in $b$. Then we have
\[
  \rmd a\wedge \rmd b  =-\frac{1}{b(\beta^2+4\alpha)}\rmd T\wedge \rmd N.
\]
Let $\gamma_{T,N}$ be the element in $G$ with trace $T$ and determinant $N$, up to conjugacy. By making the change of variable $(a,b)\mapsto (T,N)$ and note that this map is $2:1$ onto its image, \eqref{eq:weylintegration} becomes
\[
\frac{\vol(\cO_{E_{p}})}{\vol(\cO_{E_{p}}^\times)}\int_{\gamma_{T,N}\in \T(\QQ_{p})}\frac{|b||(\Delta-\overline{\Delta})^2|^{1/2}}{|N|^{1/2}}\orb(f_p;T,N)\frac{1} {|bN(\beta^2+4\alpha)|}\rmd T\rmd N.
\]
Since $|(\Delta-\overline{\Delta})^2|=|\beta^2+4\alpha|$, \eqref{eq:weylintegration} becomes
\[
\frac{\vol(\cO_{E_{p}})}{\vol(\cO_{E_{p}}^\times)}\int_{\gamma_{T,N}\in \T(\QQ_{p})}\frac{1}{|\beta^2+4\alpha|^{1/2}|N|^{3/2}}\orb(f_p;T,N)\rmd T\rmd N.
\]
By the definition of $\theta_p(T,N)$, this equals
\[
\frac{\vol(\cO_{E_{p}})}{\vol(\cO_{E_{p}}^\times)}\left(1-\frac{\omega_p}{p}\right) \int_{\gamma_{T,N}\in \T(\QQ_{p})}\frac{p^{k_\gamma} }{|\beta^2+4\alpha|^{1/2}|N|}\theta_p(T,N)\rmd T\rmd N,
\]
where \index{omegapt@$\omega_p(\T)$}$\omega_p=\omega_p(\T)=-1$ if $E_p/\QQ_p$ is inert, and $\omega_p=0$ if $E_p/\QQ_p$ is ramified. Moreover $p^{-k_\gamma}=|b|=|T^2-4N|'^{1/2}$ by Definition 2.5 of \cite{cheng2025}. Since $b^2=(T^2-4N)/(\beta^2+4\alpha)$, we obtain
\[
|T^2-4N|'^{1/2}=p^{-k_\gamma}=|b|=\frac{|T^2-4N|^{1/2}}{|\beta^2+4\alpha|^{1/2}}.
\]
Therefore
\[
|\beta^2+4\alpha|=\frac{|T^2-4N|}{|T^2-4N|'}.
\]
By the proof of \cite[Proposition 7.3]{cheng2025} we have
\[
\frac{\vol(\cO_{E_{p}})}{\vol(\cO_{E_{p}}^\times)}\left(1-\frac{\omega_p}{p}\right) =\frac{1}{1-p^{-1}}.
\]
Hence
\[
\eqref{eq:weylintegration}=\frac{1}{1-p^{-1}}\int_{\gamma_{T,N}\in \T(\QQ_{p})}\frac{1 }{|T^2-4N|^{1/2}|N|}\theta_p(T,N)\rmd T\rmd N.
\]
Finally, we make the change of variable
\[
y\mapsto 1-\frac{4N}{T^2}\quad\text{and}\quad z\mapsto T.
\]
so that $T\mapsto z$ and $N\mapsto z^2(1-y)/4$. Hence
\[
\rmd T\wedge \rmd N=\frac{z^2}{4}\rmd y\wedge\rmd z.
\]
Therefore 
\begin{align*}
\eqref{eq:weylintegration}&=\frac{1}{1-p^{-1}}\int_{\gamma_{T,N}\in \T(\QQ_{p})}\frac{1 }{|z^2y|^{1/2}|z^2(1-y)/4|}\theta_p\left(z,\frac{z^2(1-y)}{4}\right)\left|\frac{z^2}{4}\right|\rmd y\rmd z\\
&=\frac{1}{1-p^{-1}}\int_{\gamma_{T,N}\in \T(\QQ_{p})}\frac{\widehat{\Theta}_p(y)}{|1-y||y|^{1/2}}\rmd y.
\end{align*}
Hence the second assertion follows. Summing over all elliptic tori $\T$ and noting that $\gamma_{T,N}$ lies in an elliptic torus if and only if $\omega_p(y)=0$ or $-1$, we obtain the first assertion.
\end{proof}

\begin{proposition}
Let $f_p\in C_c^\infty(G)$. Then we have
\[
\widetilde{\Tr}\left(\Ind_B^G(0,\triv)(f_p)\right)=\int_{A}|D(\gamma)|^{1/2}\worb(f;\gamma)\rmd \gamma-2\int_A|D(\gamma)| ^{1/2}\log|D(\gamma)|^{1/2}\orb(f;\gamma)\rmd \gamma.
\]
\end{proposition}
\begin{proof}
The proof is similar to that of \autoref{prop:tracetilde}.
\end{proof}
\begin{proposition}
Let $f_p\in C_c^\infty(G)$. Then we have
\begin{align*}
\int_{Y_1}\log \frac{|1-y|}{|y|'}\frac{\widehat{\Theta}_p(y)}{|1-y||y|'^{1/2}}\rmd y &=\frac{\log|4|(1-p^{-1})}{2}\int_A|D(\gamma)| ^{1/2}\orb(f;\gamma)\rmd \gamma\\
&-(1-p^{-1})\int_A|D(\gamma)| ^{1/2}\log|D(\gamma)|^{1/2}\orb(f;\gamma)\rmd \gamma.
\end{align*}
\end{proposition}
\begin{proof}
By definition of $\widehat{\Theta}_p(y)$ and since $|y|'=|y|$ for $y\in Y_1$, we have
\[
\int_{Y_1}\log \frac{|1-y|}{|y|'}\frac{\widehat{\Theta}_p(y)}{|1-y||y|'^{1/2}}\rmd y=\int_{Y_1\times \QQ_p}\log\frac{|1-y|}{|y|}\theta_p\left(z,\frac{z^2(1-y)}{4}\right)\frac{\rmd y\rmd z}{|1-y||y|^{1/2}|z|}.
\]
Now we make the change of variable $z\mapsto T$ and $z^2(1-y)/4\mapsto N$ so that
\[
y\mapsto 1-\frac{4N}{T^2}\quad\text{and}\quad z\mapsto T.
\]
We have
\[
\rmd y\wedge\rmd z=\frac{4}{T^2}\rmd T\wedge \rmd N.
\]
Since $\gamma_{T,N}$ is split if and only if $y\in Y_1$, we obtain
\begin{align*}
&\int_{Y_1}\log \frac{|1-y|}{|y|'}\frac{\widehat{\Theta}_p(y)}{|1-y||y|'^{1/2}}\rmd y=\int_{\gamma_{T,N}\in A}\theta_p(T,N)\frac{\rmd T\rmd N}{|4N/T^2||T||1-4N/T^2|^{1/2}}\left|\frac{4}{T^2}\right|\log\frac{|4N/T^2|}{|1-4N/T^2|}\\
=&\int_{\gamma_{T,N}\in A}\theta_p(T,N)\frac{1}{|N||T^2-4N|^{1/2}}\log\frac{|4N|}{|T^2-4N|}\rmd T\rmd N\\
=&\log|4|\int_{\gamma_{T,N}\in A}\theta_p(T,N)\frac{\rmd T\rmd N}{|N||T^2-4N|^{1/2}} -\int_{\gamma_{T,N}\in A}\theta_p(T,N)\log\frac{|T^2-4N|}{|N|}\frac{1}{|N||T^2-4N|^{1/2}}\rmd T\rmd N.
\end{align*}

Next, we consider the right hand side. For $\gamma_{T,N}\in A$ (or equivalently, $T^2-4N\in Y_1$) we have
\[
\theta_p(\gamma)=\frac{1}{|N_\gamma|^{1/2}}\frac{1}{1-p^{-1}}|T_\gamma^2-4N_\gamma|^{1/2} \orb(f_p;\gamma),
\]
where $T_\gamma=\Tr \gamma$ and $N_\gamma=\det\gamma$.
Hence
\[
\int_A|D(\gamma)| ^{1/2}\orb(f;\gamma)\rmd \gamma=(1-p^{-1})\int_A\theta_p\begin{pmatrix}
                                                                            a &  \\
                                                                             & b 
                                                                          \end{pmatrix}\frac{\rmd a\rmd b}{|ab|}\frac{1}{(1-p^{-1})^2},             
\]
where the factor $1/(1-p^{-1})^2$ comes from the different choice of normalization..
Now we make the change of variable $T=a+b$ and $N=ab$. It is easy to compute that
\[
\rmd a\wedge\rmd b=\frac{\rmd T\wedge\rmd N}{\sqrt{T^2-4N}}
\]
and the map $(a,b)\mapsto (T,N)$ is $2:1$. Hence
\[
\int_A|D(\gamma)| ^{1/2}\orb(f;\gamma)\rmd \gamma=\frac{2}{1-p^{-1}}\int_{\gamma_{T,N}\in A}\theta_p(T,N)\frac{1}{|N||T^2-4N|^{1/2}}\rmd T\rmd N.
\]
By a similar argument we have
\[
\int_A|D(\gamma)|^{1/2}\log|D(\gamma)|^{1/2}\orb(f;\gamma)\rmd \gamma=\frac{1}{1-p^{-1}}\int_{\gamma_{T,N}\in A}\theta_p(T,N)\log\frac{|T^2-4N|}{|N|}\frac{\rmd T\rmd N}{|N||T^2-4N|^{1/2}}
\]
Hence we obtain the formula in this proposition.
\end{proof}

Now we consider the cases $p\neq 2$ and $p=2$ separately.
\subsection{The case $p\neq 2$}\label{subsec:nonarchimedean1}
 In this subsection we assume that $p\neq 2$. The main result for this subsection is the following theorem: 
\begin{theorem}\label{thm:nonarchimedean1}
Suppose that $p\neq 2$. Then the following assertions are equivalent:
\begin{enumerate}[itemsep=0pt,parsep=0pt,topsep=0pt,leftmargin=0pt,labelsep=2.5pt, itemindent=15pt,label=\upshape{(\alph*)}]
  \item \eqref{eq:nonarchimedeanresult1} holds.
  \item We have 
  \[
  I_{\A,\G}^{\G}(\triv,\gamma)=\begin{dcases}
                                 p^{-1/2}\frac{4}{1-p^{-1}}\log p, & \omega_p(\gamma)=0, \\
                                 2\frac{1+p^{-1}}{1-p^{-1}}\log p, & \omega_p(\gamma)=-1,
                               \end{dcases}\quad I_{\A,\A}^{\A}(\triv,\gamma)=1,
  \]
  and
  \[
  I_{\A,\A}^{\G}(\triv,\gamma)=-2\frac{1+p^{-1}}{1-p^{-1}}\log p+2\log |D(\gamma)|.
  \]
  where $D(\gamma)$ denotes the discriminant of $\gamma$ and $\omega_p(\gamma)=\omega_p(T_\gamma^2-4N_\gamma)$.
  \item We have
  \[
  I_{\A,\G}^{\G}(\gamma,\triv)=\begin{dcases}
                                 -p^{-1/2}\frac{4}{1-p^{-1}}\log p, & \omega_p(\gamma)=0, \\
                                 -2\frac{1+p^{-1}}{1-p^{-1}}\log p, & \omega_p(\gamma)=-1,
                               \end{dcases}\quad I_{\A,\A}^{\A}(\gamma,\triv)=1,
  \]
  and
  \[
  I_{\A,\A}^{\G}(\gamma,\triv)=-2\frac{1+p^{-1}}{1-p^{-1}}\log p+2\log |D(\gamma)|.
  \]
\end{enumerate}
\end{theorem}
\begin{proof}
(b) and (c) are equivalent as in the archimedean case. Now we prove that (a) and (b) are equivalent. 

We assume that (a) holds, namely
\begin{align*}
    &\Tr\left(R_p(0,\triv)^{-1}R_p'(0,\triv)\xi_0(f_p)\right)=-\frac{1+p^{-1}}{1-p^{-1}}\log p\Tr(\xi_0(f_p))\\
    +&\sum_{\epsilon\in \{0,-1\}}\frac{4(1-\epsilon p^{-1})\log p}{(1-\epsilon)(1-p^{-1})^2}\int_{Y_{\epsilon}}|1-y|_{p}^{-1}\widehat{\Theta}_{p}(y) |y|_{p}'^{-\frac12}\rmd y\\
   -&\frac{4}{1-p^{-1}}\int_{Y_{1}}\log\frac{|1-y|_{p}}{|y|_{p}'}|1-y|_{p}^{-1}\widehat{\Theta}_{p}(y) |y|_{p}'^{-\frac12}\rmd y
   +\widetilde{\Tr}\left(\Ind_{\B(\QQ_p)}^{\G(\QQ_p)}(0,\triv)(f_p)\right).
\end{align*}

By (the local version of) \cite[Proposition 8.2]{cheng2025}, we have
\[
\Tr(\xi_0(f_p))=\int_{A}\frac{|a-b|}{|ab|^{1/2}}\orb(f_p;\gamma)\rmd \gamma=\int_{A}|D(\gamma)|^{1/2}\orb(f_p;\gamma)\rmd \gamma.
\]
Recall that for $p\neq 2$, we have $|y|'=|y|$ if $\omega_p(y)=-1$ and $|y|'=p|y|$ if $\omega_p(y)=0$. Hence
\[
\int_{Y_\epsilon}\frac{\widehat{\Theta}_p(y)}{|1-y||y|^{1/2}}\rmd y=\begin{dcases}
    \int_{Y_\epsilon}\frac{\widehat{\Theta}_p(y)}{|1-y||y|'^{1/2}}\rmd y, & \epsilon=-1,\\
    p^{1/2}\int_{Y_\epsilon}\frac{\widehat{\Theta}_p(y)}{|1-y||y|'^{1/2}}\rmd y, & \epsilon=0.
  \end{dcases}
\]
Now using \autoref{thm:normalizedintertwining} and substituting the results of this section into the above formula, we conclude that
\begin{align*}
 &\int_{\Gamma_\el(G)}I_{\A,\G}^\G(\triv,\gamma)|D(\gamma)|^{1/2}\orb(f;\gamma)\rmd \gamma +\int_{A}I_{\A,\A}^\A(\triv,\gamma)|D(\gamma)|^{1/2}\worb(f;\gamma)\rmd \gamma 
 \\+ & \frac12\int_{A}I_{\A,\A}^\G(\triv,\gamma)|D(\gamma)|^{1/2}\orb(f;\gamma)\rmd \gamma\\
 =&-\frac{1+p^{-1}}{1-p^{-1}}\log p\int_{A}|D(\gamma)|^{1/2}\orb(f;\gamma)\rmd \gamma+\frac{2p^{-1/2}\log p}{1-p^{-1}}\sum_{\T\colon\omega_p(\T)=0}\int_{\T(\QQ_p)}|D(\gamma)|^{1/2} \orb(f;\gamma)\rmd \gamma\\
 +&\frac{(1+p^{-1})\log p}{1-p^{-1}}\sum_{\T\colon\omega_p(\T)=-1}\int_{\T(\QQ_p)}|D(\gamma)|^{1/2} \orb(f;\gamma)\rmd \gamma\\
 +&4\int_A|D(\gamma)| ^{1/2}\log|D(\gamma)|^{1/2}\orb(f;\gamma)\rmd \gamma\\
 +&\int_{A}|D(\gamma)|^{1/2}\worb(f;\gamma)\rmd \gamma-2\int_A|D(\gamma)| ^{1/2}\log|D(\gamma)|^{1/2}\orb(f;\gamma)\rmd \gamma.\end{align*}
Hence (b) holds by uniqueness of the kernel function. (b) also implies (a) by reversing the proof above.
\end{proof}

\subsection{The case $p=2$} 
The $p=2$ case is a bit subtle. We need to define a new function $\lambda_p$ on the discriminant as 
\[
\index{lambdapy@$\lambda_p(y)$}\lambda_p(y)=\begin{cases}
               1, & v_p(y)\text{ is odd}, \\
               0, & v_p(y)\text{ is even},
             \end{cases}
\]
and we define \index{y00@$Y_{0,0}$}\index{y01@$Y_{0,1}$}
\[
Y_{0,0}=\{y\in Y_0\,|\,\lambda_p(y)=0\}\quad\text{and}\quad Y_{0,1}=\{y\in Y_0\,|\,\lambda_p(y)=1\}.
\]
Note that $|y|'=p^2|y|$ for $y\in Y_{0,0}$ and $|y|'=p^3|y|$ for $y\in Y_{0,1}$.

We define $\lambda_p(\gamma)$ as $\lambda_p(T_\gamma^2-4N_\gamma)$ for any regular element $\gamma\in G$.
We also define \index{lambdapt@$\lambda_p(\T)$}$\lambda_p(\T)$ for elliptic tori $\T$ of $G$ as $\lambda_p(\gamma)$, where $\gamma\in \T(\QQ_p)$ is a regular element (which does not depend on the choice of $\gamma$). The main theorem of this subsection is the following:
\begin{theorem}\label{thm:nonarchimedean2}
Suppose that $p= 2$. Then the following assertions are equivalent:
\begin{enumerate}[itemsep=0pt,parsep=0pt,topsep=0pt,leftmargin=0pt,labelsep=2.5pt, itemindent=15pt,label=\upshape{(\alph*)}]
  \item \eqref{eq:nonarchimedeanresult2} holds.
  \item We have 
  \[
  I_{\A,\G}^{\G}(\triv,\gamma)=\begin{dcases}
                                 p^{-1}\frac{4}{1-p^{-1}}\log p, & \omega_p(\gamma)=0,\, \lambda_p(\gamma)=0, \\
                                 p^{-3/2}\frac{4}{1-p^{-1}}\log p, & \omega_p(\gamma)=0,\, \lambda_p(\gamma)=1,\\
                                 2\frac{1+p^{-1}}{1-p^{-1}}\log p, & \omega_p(\gamma)=-1,
                               \end{dcases}\quad I_{\A,\A}^{\A}(\triv,\gamma)=1,
  \]
  and
  \[
  I_{\A,\A}^{\G}(\triv,\gamma)=-2\frac{1+p^{-1}}{1-p^{-1}}\log p+2\log |D(\gamma)|.
  \]
  where $D(\gamma)$ denotes the discriminant of $\gamma$ and $\omega_p(\gamma)=\omega_p(T_\gamma^2-4N_\gamma)$.
  \item We have
  \[
  I_{\A,\G}^{\G}(\gamma,\triv)=\begin{dcases}
                                 -p^{-1}\frac{4}{1-p^{-1}}\log p, & \omega_p(\gamma)=0,\, \lambda_p(\gamma)=0, \\
                                 -p^{-3/2}\frac{4}{1-p^{-1}}\log p, & \omega_p(\gamma)=0,\, \lambda_p(\gamma)=1,\\
                                 -2\frac{1+p^{-1}}{1-p^{-1}}\log p, & \omega_p(\gamma)=-1,
                               \end{dcases}\quad I_{\A,\A}^{\A}(\gamma,\triv)=1,
  \]
  and
  \[
  I_{\A,\A}^{\G}(\gamma,\triv)=-2\frac{1+p^{-1}}{1-p^{-1}}\log p+2\log |D(\gamma)|.
  \]
\end{enumerate}
\end{theorem}

\begin{proof}
(b) and (c) are equivalent as in the previous cases. Now we prove that (a) and (b) are equivalent. 

We assume that (a) holds, namely
\begin{align*}
    &\Tr\left(R_p(0,\triv)^{-1}R_p'(0,\triv)\xi_0(f_p)\right)=-4\log p\Tr(\xi_0(f_p))-\frac{1+p^{-1}}{1-p^{-1}}\log p\Tr(\xi_0(f_p))\\
    +&\sum_{\epsilon\in \{0,-1\}}\frac{4(1-\epsilon p^{-1})\log p}{(1-\epsilon)(1-p^{-1})^2}\int_{Y_{\epsilon}}|1-y|_{p}^{-1}\widehat{\Theta}_{p}(y) |y|_{p}'^{-\frac12}\rmd y\\
   -&\frac{4}{1-p^{-1}}\int_{Y_{1}}\log\frac{|1-y|_{p}}{|y|_{p}'}|1-y|_{p}^{-1}\widehat{\Theta}_{p}(y) |y|_{p}'^{-\frac12}\rmd y
   +\widetilde{\Tr}\left(\Ind_{\B(\QQ_p)}^{\G(\QQ_p)}(0,\triv)(f_p)\right).
\end{align*}
In this case, we still have 
\[
\Tr(\xi_0(f_p))=\int_{A}\frac{|a-b|}{|ab|^{1/2}}\orb(f_p;\gamma)\rmd \gamma=\int_{A}|D(\gamma)|^{1/2}\orb(f_p;\gamma)\rmd \gamma
\]
and by definition, for $\epsilon=-1$ we have
\[
\int_{Y_\epsilon}\frac{\widehat{\Theta}_p(y)}{|1-y||y|^{1/2}}\rmd y=
    \int_{Y_\epsilon}\frac{\widehat{\Theta}_p(y)}{|1-y||y|'^{1/2}}\rmd y
\]
and
\[
    \int_{Y_{\epsilon,\lambda}}\frac{\widehat{\Theta}_p(y)}{|1-y||y|^{1/2}}\rmd y=\begin{dcases} p\int_{Y_{\epsilon,\lambda}}\frac{\widehat{\Theta}_p(y)}{|1-y||y|'^{1/2}}\rmd y, & \epsilon=0,\,\lambda=0,\\
    p^{3/2}\int_{Y_{\epsilon,\lambda}}\frac{\widehat{\Theta}_p(y)}{|1-y||y|'^{1/2}}\rmd y, & \epsilon=0,\,\lambda=1.
  \end{dcases}
\]
Using \autoref{thm:normalizedintertwining} and substituting the results of this section into the above formula, we conclude that
\begin{align*}
 &\int_{\Gamma_\el(G)}I_{\A,\G}^\G(\triv,\gamma)|D(\gamma)|^{1/2}\orb(f;\gamma)\rmd \gamma +\int_{A}I_{\A,\A}^\A(\triv,\gamma)|D(\gamma)|^{1/2}\worb(f;\gamma)\rmd \gamma 
 \\+ & \frac12\int_{A}I_{\A,\A}^\G(\triv,\gamma)|D(\gamma)|^{1/2}\orb(f;\gamma)\rmd \gamma=-\left(4+\frac{1+p^{-1}}{1-p^{-1}}\right)\log p\int_{A}|D(\gamma)|^{1/2}\orb(f_p;\gamma)\rmd \gamma\\
 +&\frac{(1+p^{-1})\log p}{1-p^{-1}}\sum_{\T\colon\omega_p(\T)=-1}\int_{\T(\QQ_p)}|D(\gamma)|^{1/2} \orb(f_p;\gamma)\rmd \gamma\\
 +&\frac{2p^{-1}\log p}{1-p^{-1}}\sum_{\T\colon\omega_p(\T)=0,\,\lambda_p(\T)=0}\int_{\T(\QQ_p)}|D(\gamma)|^{1/2} \orb(f_p;\gamma)\rmd \gamma\\
 +&\frac{2p^{-3/2}\log p}{1-p^{-1}}\sum_{\T\colon\omega_p(\T)=0,\,\lambda_p(\T)=1}\int_{\T(\QQ_p)}|D(\gamma)|^{1/2} \orb(f_p;\gamma)\rmd \gamma\\
 +&4\int_A|D(\gamma)| ^{1/2}\log|D(\gamma)|^{1/2}\orb(f;\gamma)\rmd \gamma
 +4\log2\int_A|D(\gamma)| ^{1/2}\orb(f;\gamma)\rmd \gamma\\
 +&\int_{A}|D(\gamma)|^{1/2}\worb(f;\gamma)\rmd \gamma-2\int_A|D(\gamma)| ^{1/2}\log|D(\gamma)|^{1/2}\orb(f;\gamma)\rmd \gamma.
\end{align*}
Hence (b) holds by uniqueness of the kernel function. (b) also implies (a) by reversing the proof above.
\end{proof}

\section{The desired cancellation}\label{sec:comparison}
In this section we compare with the known results to derive the kernel functions $I_{\A,\A}^{\A}(\gamma,\triv)$, $I_{\A,\A}^{\G}(\gamma,\triv)$, and $I_{\A,\G}^{\G}(\gamma,\triv)$ and see that they have the expected values. From these results we derive the main result \eqref{eq:standardrepresentation}.

\subsection{The archimedean case}
As presented in the introductory section, the computation of $I_{\L,\M}^{\S}(\gamma,\tau)$ was considered in \cite{arthur1985,hoffmann1997,hoffmann2008,hoffmann2012} for lower rank groups. In this subsection we will mainly use the result of \cite[Section 9]{hoffmann2008}.

Let $G=\GL_2(\RR)$, and $A=\A(\RR)$ be the diagonal torus of $G$. At the real place we work modulo the positive scalar center $Z_+$. The Weyl discriminant factor is
\[
D(t)=\frac{(a_1-a_2)^2}{a_1a_2}.
\]
The elliptic torus is identified with $\CC^\times$, and the quotient measure on $Z_+\bs\CC^\times\cong \SS^1$ is $2\rmd\theta$. The positive root is $\alpha(\begin{smallmatrix}a_1 & \\ & a_2\end{smallmatrix})=a_1/a_2$.

In Hoffmann's notation, the scalar $n_\A^\G$ is the length of the coroot $\alpha\spcheck$ with respect to the chosen measure on $\mathfrak a_A^G$.  In the above quotient-measure normalization this length is
\[
n_\A^\G=2.
\]
The sign convention for $I_{\A,\A}^{\G}$ below is the one used in the local weighted identity of the direct proof, i.e. with the weight $\alpha(H_B(wg)+H_B(g))$.  With this convention the split kernel is the negative of Hoffmann's displayed principal series coefficient in Theorem 9.1(iii).

\subsubsection{The Levi kernel $I_{\A,\A}^{\A}$}
The function $I_{\A,\A}^{\A}(\tau,\gamma)$ for $\tau=\triv$ can be easily computed. Indeed, by \autoref{thm:normalizedintertwining} we have
\[
\Tr(\triv(g))=\int_{A}I_{\A,\A}^{\A}(\triv,\gamma)g(\gamma)\rmd\gamma.
\]
By definition we have
\[
\Tr(\triv(g))=\int_{A}g(\gamma)\rmd\gamma.
\]
Hence $I_{\A,\A}^{\A}(\triv,\gamma)=1$.

\subsubsection{The elliptic kernel $I_{\A,\G}^{\G}$}

\begin{proposition}
For every elliptic regular $\gamma\in\GL_2(\RR)$,
\[
I_{\A,\G}^{\G}(\triv,\gamma)=2\uppi.
\]
\end{proposition}

\begin{proof}
This is the $\GL_2$ specialization of Hoffmann's Theorem 9.1(ii).  We recall
the limiting calculation.  Let
\[
\tau_s\begin{pmatrix}a_1 & \\ & a_2\end{pmatrix}=\left|\frac{a_1}{a_2}\right|^s,\qquad s\to0.
\]
In Hoffmann's notation this means $\lambda_1=s$, $\lambda_2=-s$, so
$\lambda_{12}=2s$ and $\lambda_{21}=-2s$, up to the harmless conversion
between $s$ and the purely imaginary parameter in the unitary axis.

If $\gamma$ is elliptic, its eigenvalues in $\CC$ are conjugate.  The
principal branch convention in Hoffmann's formula gives the two factors
\[
1-\rme^{\dpii s},\qquad 1-\rme^{-\dpii s}.
\]
Thus the elliptic kernel for $\tau_s$ has the form
\[
I_{\A,\G}^{\G}(\tau_s,\gamma)
=\frac{n_\A^\G}{2}
\left(
\eta_s\frac{1-\rme^{\dpii s}}{\sin(2s/\rmi)}
+\eta_s^{-1}\frac{1-\rme^{-\dpii s}}{\sin(-2s/\rmi)}
\right),
\]
where $\eta_s\to 1$ as $s\to 0$. The two limits are
\[
\frac{1-\rme^{\dpii s}}{\sin(2s/\rmi)}\to \uppi,
\qquad
\frac{1-\rme^{-\dpii s}}{\sin(-2s/\rmi)}\to \uppi.
\]
as $s\to 0$. Consequently,
\[
I_{\A,\G}^{\G}(\triv,\gamma)=n_\A^\G\uppi=2\uppi.
\]
This agrees with the Arthur-Herb-Sally computation for the real rank-one elliptic kernel, after passing from $\SL_2(\RR)$ to $Z_+\bs\GL_2(\RR)$ with quotient measure $2\rmd\theta$.
\end{proof}

\subsubsection{The split weighted kernel $I_{\A,\A}^{\G}$}
Finally we consider the function $I_{\A,\A}^{\G}$.
\begin{proposition}
For $t=(\begin{smallmatrix}a_1&\\&a_2\end{smallmatrix})\in A_{\rm reg}$,
\[
 I_{\A,\A}^{\G}(\triv,t)=-4\log2+2\log|D(t)|.
\]
\end{proposition}

\begin{proof}
This is the principal-series part of \cite[Theorem 9.1(iii)]{hoffmann2008}, translated to
the sign convention of the weighted orbital integral used here.  Assume first
that $|a_1|>|a_2|$ and put
\[
r=\left|\frac{a_1}{a_2}\right|,\qquad z=\frac{a_2}{a_1}.
\]
Hoffmann's formula tells us that the kernel $I_{\A,\A}^{\G}(\tau_s,t)$ is $n_\A^\G$ times
 \[
\tau_s(t)b(-2s,z)+\tau_s(t)u(s)+(w\tau_s)(t)b(2s,z)+(w\tau_s)(t)u(-s),
\]
where $\tau_s$ is as in the previous proposition,
$b(s,z)$ is given by the series
\[
b(s,z)=\sum_{n=1}^{+\infty}\frac{z^n}{n+s}
\]
and
\[
u(s)=\frac12\psi(s)-\frac12\psi\left(s+\frac12\right).
\]
Here \index{psi@$\psi(s)$}$\psi(s)=\Gamma'(s)/\Gamma(s)$.  We take the limit $s\to 0$.

The $b$-terms are regular at $s=0$, and give
\[
b(0,z)+b(0,z)=2b(0,z)=-2\log(1-z).
\]
For the $u$-terms, we use
\[
\psi(s)=-\frac1s-\upgamma+O(s),\quad\text{and}\quad\psi\left(\frac12\right)=-\upgamma-2\log2.
\]
Hence
\[
u(s)=-\frac{1}{2s}+\log2+O(s),\qquad u(-s)=\frac{1}{2s}+\log2+O(s).
\]
Since $r^s=1+s\log r+O(s^2)$ and $r^{-s}=1-s\log r+O(s^2)$, we get
\[
r^su(s)=-\frac{1}{2s}-\frac12\log r+\log2+O(s),
\]
\[
r^{-s}u(-s)=\frac{1}{2s}-\frac12\log r+\log2+O(s).
\]
Their poles cancel, leaving
\[
r^s u(s)+r^{-s}u(-s)=-\log r+2\log2+O(s).
\]
Therefore the limiting bracket is
\[
-2\log(1-z)-\log r+2\log2.
\]

Now
\[
|D(t)|=\frac{|a_1-a_2|^2}{|a_1a_2|}=r|1-z|^2,
\]
so
\[
-2\log|1-z|-\log r=-\log|D(t)|.
\]
Thus Hoffmann's principal-series coefficient is
\[
n_\A^\G\left(-\log|D(t)|+2\log2\right).
\]
With $n_\A^\G=2$, and with the sign convention for the kernel in the weighted identity used here, this gives
\[
I_{\A,\A}^{\G}(\triv,t)=-2\left(-\log|D(t)|+2\log2\right)=2\log|D(t)|-4\log2.
\]
This is the asserted formula. The formula for $|a_1|<|a_2|$ follows
by Weyl invariance, replacing $t$ by $wtw^{-1}$.
\end{proof}

Combining the above we obtain (b) in \autoref{thm:archimedean}. Hence we have solved the archimedean part.

\subsection{The nonarchimedean case}
The $p$-adic computation of $I_{\A,\G}^\G(\triv,\gamma)$ was carried out with assistance from ChatGPT and we include the details in \autoref{sec:explicitnonarchimedean}. By the results in \autoref{subsec:explicitqp} we obtain assertion (b) of \autoref{thm:nonarchimedean1} and \autoref{thm:nonarchimedean2}. Hence we have solved the nonarchimedean part. 

Combining these results, in this series of five papers we have proved that
\[
\lim_{X\to +\infty}\frac{1}{X}\sum_{\substack{n<X\\ \gcd(n,S)=1}}I_\cusp(f^n)=0
\]
without any additional assumptions.

\appendix

\section{Direct computation of Arthur's $p$-adic kernel over $\GL_2$ for $\tau=\triv$}\label{sec:explicitnonarchimedean}

\subsection{Statement of the result}\label{subsec:measure}
Let $F$ be a $p$-adic field with ring of integers $\cO$, uniformizer $\varpi$, and residue field of cardinality $q$. Let $v$ be the valuation on $F$ and $|\cdot|$ be the normalized absolute value on $F$ (i.e.,$|\varpi|=q^{-1}$). The additive Haar measure on $F$ is normalized such that $\vol(\cO)=1$.  Normalize Haar measure on $\G(F)$ by $\vol(K)=1$ and use Arthur's quotient measures on tori and regular centralizers \cite[pp.~168--169]{arthur1994}.  Set $\G=\GL_2$, $\A$ the diagonal torus of $\G$, $\B$ (resp. $\overline{\B}$) the subgroup of upper (resp. lower) triangular matrices. Let
\[
  G=\G(F),\qquad A=\A(F), \qquad B=\B(F),\qquad K=\GL_2(\cO).
\]
Arthur's spectral variable is an essential triplet, not merely a character. In this appendix
\[
  \triv:=(A,\triv_A,1)\in T_{\mathrm{disc}}(\A)
\]
denotes the triplet with spectral Levi $\A$, trivial unitary character of $\A(F)$, and trivial finite $R$-group datum. For a regular semisimple element $\gamma$, write
\[
  D(\gamma)=
  \det\left(1-\operatorname{Ad}(\gamma)
       \middle|_{\mathfrak g/\mathfrak g_\gamma}\right).
\]
Thus, for $t=\operatorname{diag}(a,b)$ and $r=b/a$,
\[
  |D(t)|=|(1-r)(1-r^{-1})|.
\]

We use normalized parabolic induction
\[
  \index{pis@$\pi_s$}\pi_s=\Ind_B^G(\chi_s),\qquad
  \index{chis@$\chi_s$}\chi_s(\operatorname{diag}(a,d))=|a/d|^s.
\]
The space \index{hs@$\cH_s$}$\cH_s$ of $\pi_s$ is
\[
\cH_s=\left\{f\in C^\infty(G)\,\middle|\, f(xg)=\left|\frac{a}{d}\right|^{s+1/2}f(g)\text{ for all }x=\begin{pmatrix} a & b \\ 0 & d \end{pmatrix}\in B\right\}
\]
and $\pi_s$ acts on $\cH_s$ by $\pi_s(g)f(x)=f(xg)$.

The \emph{intertwining operator} (cf. \cite[(38)]{schmidt2002some} for example) is defined by
\[
\index{ms@$M(s)$}M(s):\cH_s\to \cH_{-s}, \qquad M(s)f(g)=\int_Ff\left(\begin{pmatrix} 0 & -1\\ 1 & 0 \end{pmatrix} \begin{pmatrix} 1 & u\\ 0 & 1 \end{pmatrix} g\right)\rmd u
\]
for $\Re s>0$ and extended by meromorphic continuation. The \emph{normalized intertwining operator} is defined by \index{rs@$R(s)$}$R(s)=m(s)^{-1}M(s)$ with
\[
\index{msmalls@$m(s)$}m(s)=\frac{1-q^{-2s-1}}{1-q^{-2s}}.
\]
$R(s)$ can be characterized by
\[
  R(s)\triv=\triv,\qquad R(-s)R(s)=1.
\]
\begin{theorem}\label{thm:raw-main}
The values of Arthur's $p$-adic kernels for $\GL_2$ over $F$ are given as follows: 
\begin{enumerate}[itemsep=0pt,parsep=0pt,topsep=0pt, leftmargin=0pt,labelsep=2.5pt,itemindent=15pt,label=\upshape{(\arabic*)}]
\item For $t\in \A(F)_\reg$,
\begin{equation}\label{eq:raw-split}
  I_{\A,\A}^{\G}(\triv,t)=-2\frac{1+q^{-1}}{1-q^{-1}}\log q+2\log|D(t)|.
\end{equation}
\item Let $\gamma\in \G(F)_\reg$ be elliptic, let $E_\gamma=F(\sqrt{\Delta_\gamma})$ be its quadratic centralizer, and let $\delta(E_\gamma/F)$ be the exponent of the discriminant ideal.  Then
\begin{equation}\label{eq:raw-elliptic}
  I_{\A,\G}^{\G}(\triv,\gamma)=
  \begin{dcases}
  2\frac{1+q^{-1}}{1-q^{-1}}\log q,
       &E_\gamma/F\text{ unramified},\\[7pt]
  \frac{4q^{-\delta(E_\gamma/F)/2}}{1-q^{-1}}\log q,
       &E_\gamma/F\text{ ramified}.
  \end{dcases}
\end{equation}
\item For the ambient torus itself,
\begin{equation}\label{eq:torus-kernel}
  I_{\A,\A}^\A(\triv,t)=1.
\end{equation}
\end{enumerate}
\end{theorem}

The rest of the appendix proves these statements directly and explains exactly
where Sections 6, 8, and 9 of \cite{arthur1994} enter.

\subsection{What Sections 6, 8, and 9 provide}

Fix an elliptic maximal torus $\T$ in a geometric Levi subgroup $\M$, and
take $\theta\in C_c^\infty(\T(F)_\reg)$.  Arthur defines
\[
  \phi_\theta(\gamma)=\sum_{w\in W(\G,\T)}\theta(w\gamma),
\]
and chooses $g_\theta\in C_c^\infty(\G(F)_\reg)$ with invariant orbital transform $\mathcal T_\G g_\theta=\phi_\theta$. The smeared distribution is
\[
  I_\M(\theta,f):=\int_T\theta(\gamma)I_\M(\gamma,f)\rmd\gamma=\frac{1}{|W(\G,\T)|}
    \int_T\phi_\theta(\gamma)I_\M(\gamma,f)\rmd\gamma.
\]
With this notation, Arthur's equation (6.5) is
\begin{equation}\label{eq:Arthur65}
  I_\M(\theta,\tau)
  =(-1)^{\dim(A_\M\times A_\L)}i^\L(\tau)
    I_\L(\tau^\vee,g_\theta).
\end{equation}
This defines a smeared coefficient, not yet a pointwise value. In the specialization used below, $\L=\A$ and $\tau=\triv$.  The relative Weyl group of $A$ in the ambient group $A$ is trivial, so the regular Weyl sum in Arthur's definition of $i^A(\tau)$ consists only of the identity and its determinant factor is one.  Hence
\[
i^\A(\triv)=1.
\]

For orientation, Arthur's dual kernel expansion (Theorem~4.3)
specializes in this rank-one situation to
\[
 I_\A^\G(\triv,g)
 =\frac12\int_{\A(F)_\reg}
 I_{\A,\A}^{\G}(\triv,t)g_\A(t)\rmd t
 +\int_{\Gamma_\el(\G(F))}
 I_{\A,\G}^{\G}(\triv,\gamma)g_\G(\gamma)\rmd\gamma.
\]
The factor $1/2$ is the exterior Weyl coefficient on the split stratum; it is not the value of either kernel.  The local Weyl sums used below cancel this factor before the density is read off.

For $p$-adic $F$, Theorem 8.1 of \cite{arthur1994} proves that, after fixing a compact open subgroup and a compact open subset of $\T(F)$, the relevant distributions span a finite-dimensional space.  Fix now a compact open $\Delta\subseteq \T(F)$, a spectral component $\Omega$, and Arthur's finite index set $\mathcal A(\Delta,\Omega)$.  On $C(\G(F),\Omega)$, \cite[Section 9]{arthur1994} chooses a basis $\{I_\alpha\,|\,\alpha\in\mathcal A(\Delta,\Omega)\}$ and writes, for $\gamma\in\Delta_\reg$ and $f\in C(\G(F),\Omega)$,
\begin{equation}\label{eq:finite-basis}
  I_\M(\gamma,f)=\sum_\alpha t_\M^\alpha(\gamma)I_\alpha(f).
\end{equation}
He then chooses functions $\theta_\beta\in C_c^\infty(\Delta_\reg)$ dual
to the coefficient functions,
so that
\[
  \int_\Delta\theta_\beta(\gamma)t_\M^\alpha(\gamma)\rmd\gamma
  =\delta_{\alpha\beta},
\]
where $\delta_{\alpha\beta}$ is the Kronecker symbol. Then he defines, in equation (9.2), for $\tau\in T_{\mathrm{disc}}(\L)\cap\Omega_\L$,
\begin{equation}\label{eq:Arthur92}
  I_\M(\gamma,\tau)=\sum_\alpha t_\M^\alpha(\gamma)I_\M(\theta_\alpha,\tau).
\end{equation}

Here is the exact finite-dimensional implication that will be used below.

\begin{theorem}\label{thm:smeared-to-pointwise}
Fix the compact set $\Delta$ and spectral component $\Omega$ used in \cite[Section 9]{arthur1994}.  Suppose that a locally constant function $F_\tau$ on $\Delta_\reg$ satisfies
\[
  I_\M(\theta,\tau)=\int_\Delta\theta(\gamma)F_\tau(\gamma)\rmd\gamma
\]
for every $\theta\in C_c^\infty(\Delta_\reg)$.  Then $F_\tau$ is the
pointwise kernel defined by \eqref{eq:Arthur92}.
\end{theorem}

\begin{proof}
Put
\[
  \cN=\left\{\theta\ \middle|\ 
  \int_\Delta\theta(\gamma)t_M^\alpha(\gamma)\rmd\gamma=0
  \text{ for every }\alpha\right\}.
\]
Integrating \eqref{eq:finite-basis} against $\theta$ shows that $I_
\M(\theta,f)=0$ for every $\theta\in \cN$ and every $f\in C(\G(F),\Omega)$.  At this point we use exactly the coefficient-separation step in Arthur's construction, rather than asserting a new Paley--Wiener theorem: after Lemma 6.1, Arthur proves in \cite[pp.209--210]{arthur1994} that the spectral coefficient attached to a fixed component depends only on these finitely many geometric coefficients. Applying this to the zero coefficient vector above and using the uniqueness of the expansion (4.1) on the $\G$-regular locus, Arthur's Remark 1 after Theorem 4.5, together with the argument on \cite[pp. 209--210]{arthur1994} gives
\[
I_\M(\theta,\tau)=0
\]
for every $\G$-regular $\tau\in T_{\mathrm{disc}}(\L)\cap\Omega_\L$.  This is not yet enough at the trivial datum $\triv$.  Arthur's \cite[Lemma 6.2]{arthur1994} says that $\tau\mapsto I_\M(\theta,\tau)$ is smooth, while the $\G$-regular locus is dense in each spectral component.  Hence the same vanishing holds on the whole of $T_{\mathrm{disc}}(\L)\cap\Omega_\L$, including $\triv$.  Thus the functional $\theta\mapsto I_\M(\theta,\tau)$ factors through the finite-dimensional quotient by $\cN$. The duality relation for the $\theta_\alpha$ gives, for every $\theta\in C_c^\infty(\Delta_\reg)$,
\[
  I_\M(\theta,\tau)= \sum_\alpha\left(\int_\Delta\theta(\gamma)t_M^\alpha(\gamma)\rmd\gamma\right)I_\M(\theta_\alpha,\tau).
\]
Comparison with \eqref{eq:Arthur92} shows that
\[
\int_\Delta\theta(\gamma)F_\tau(\gamma)\rmd\gamma= \int_\Delta\theta(\gamma)I_\M(\gamma,\tau)\rmd\gamma
\]
for every $\theta\in C_c^\infty(\Delta_\reg)$. Since the functions $F_\tau(\gamma)$ and $I_\M(\gamma,\tau)$ are locally constant, we find that they are equal.
\end{proof}

\begin{lemma}\label{lem:local-section}
Let $\T=\A$, or let $\T$ be an elliptic maximal torus. Fix $\gamma_0\in \T(F)_\reg$.  There is a compact open set $U\subset \T(F)_\reg$, containing $\gamma_0$ and disjoint from its nontrivial
Weyl translate, such that for every $\theta\in C_c^\infty(U)$ one can choose $g_\theta$ in \eqref{eq:Arthur65} with the following properties:
\begin{enumerate}[itemsep=0pt,parsep=0pt,topsep=0pt, leftmargin=0pt,labelsep=2.5pt,itemindent=15pt,label=\upshape{(\roman*)}]
\item $\mathcal T_\G g_\theta=\phi_\theta$;
\item if $\T=\A$, every conjugating variable in the support of $g_\theta$ has a representative in $K$, and Arthur's rank-one geometric weight is zero there;
\item if $\T$ is elliptic, all split orbital transforms of $g_\theta$ vanish.
\end{enumerate}
Consequently, the proper Levi subtraction in the dual Definition 3.2 of \cite{arthur1994} vanishes on $g_\theta$:
\begin{equation}\label{eq:subtraction-vanishes}
  I_\A(\triv,g_\theta)=J_\A(\triv,g_\theta).
\end{equation}
\end{lemma}

\begin{proof}
Consider the analytic submersion
\[
  (\T(F)\cap K)\backslash K\times \T(F)_\reg
  \rightarrow \G(F)_\reg,
  \qquad (k,t)\mapsto k^{-1}tk.
\]
This is the rank-one instance of Harish-Chandra's submersion principle \cite[pp.~95--102]{harish1993submersion}.  After shrinking a compact open neighborhood $U$ of $\gamma_0$, disjoint from its Weyl translate, and restricting the first factor to a compact open neighborhood $V$ of the identity coset, the map is a homeomorphism onto a compact open image. 
Give $(T\cap K)\backslash K$ the quotient of the fixed Haar measures, and choose
$\eta\in C_c^\infty(V)$ with $\int_V\eta(k)\rmd k=1$.  Define $g_\theta$
on the image by
\[
  g_\theta(k^{-1}tk)
  =|D(t)|^{-1/2}\theta(t)\eta(k),
\]
and define it to be zero outside the image.  By the definition of the
normalized orbital transform,
\[
  (g_\theta)_\G(t)=|D(t)|^{1/2}\int_{\T(F)\backslash \G(F)}g_\theta(x^{-1}tx)\rmd x=\theta(t)
\]
for $t\in U$.  The same conjugacy class meets $T$ also in $wt$, and $U\cap wU=\varnothing$; hence this identity is exactly $(g_\theta)_\G=\theta+\theta\circ w=\phi_\theta$. This proves (i) with
all measure factors displayed.

For $\T=\A$, every contributing conjugator has a representative $k\in K$. Then
\[
  H_\B(k)=H_{\overline \B}(k)=0.
\]
Arthur's $v_\A(k)$ is the normalized length of the convex hull of these two points, and is therefore zero.  More explicitly, every conjugator $x$ for which $g_\theta(x^{-1}tx)\ne0$ has a representative in $K$ (the second Weyl representative also lies in $K$), and therefore
\[
 J_\A^\G(t,g_\theta)=|D(t)|^{1/2}
 \int_{\A(F)\backslash \G(F)}g_\theta(x^{-1}tx)v_\A(x)\rmd x=0.
\]
For split $a$ outside the chosen conjugacy classes the same weighted transform is zero by support.  Hence $\phi_\A(g_\theta)(a)=J_\A^\G(a,g_\theta)$ vanishes for every regular $a\in \A(F)$, not only for $a\in U$. For elliptic support, a sufficiently small neighborhood contains no split regular elements, so every split orbital transform is zero.  These assertions give $I_\A^\A(\triv,\phi_\A(g_\theta))=0$. Hence \eqref{eq:subtraction-vanishes} is derived from \cite[Definition 3.2]{arthur1994}.
\end{proof}

\subsection{The normalized intertwining operator on $\mathbf P^1(F)$}\label{subsec:intertwining}
Recall that
\[
\index{hs@$\cH_s$}\cH_s=\left\{f\in C^\infty(G)\,\middle|\, f(xg)=\left|\frac{a}{d}\right|^{s+1/2}f(g)\text{ for all }x=\begin{pmatrix} a & b \\ 0 & d \end{pmatrix}\in B\right\}.
\]

The elements in $\cH_s$ can be considered as functions on $B\bs G\cong\bP^1(F)$. In fact, we have the following identification
\[
\index{ws@$W_s$}W_s:\cH_s\cong C^\infty(B\bs G)\cong \index{hsc@$\cH_\mathrm{c}$} \cH_\mathrm{c}:=C^\infty(\bP^1(F)),
\]
where the explicit formula for $W_s$ will be clear later.

The first goal in this section is to give an explicit formula for the intertwining operator \index{rsc@$R_{\mathrm c}(s)$} $R_{\mathrm c}(s)$ such that the diagram
\[
\begin{tikzcd}
	{\mathcal{H}_s} && {\mathcal{H}_\mathrm{c}} \\
	\\
	{\mathcal{H}_{-s}} && {\mathcal{H}_\mathrm{c}}
	\arrow["{W_s}"', from=1-1, to=1-3]
	\arrow["{R(s)}", from=1-1, to=3-1]
	\arrow["{R_{\mathrm c}(s)}"', from=1-3, to=3-3]
	\arrow["{W_{-s}}", from=3-1, to=3-3]
\end{tikzcd}
\]
commutes. 

Let \index{nu@$n(u)$} \index{nubar@$\overline{n}(u)$} \index{w@$w$} \index{rx@$r(x)$}
\[
n(u)=
\begin{pmatrix}1&u\\0&1\end{pmatrix},
\qquad
\overline n(x)=
\begin{pmatrix}1&0\\x&1\end{pmatrix},
\qquad
w=
\begin{pmatrix}0&-1\\1&0\end{pmatrix},
\qquad r(x)=wn(-x)=\begin{pmatrix}0&-1\\1&-x\end{pmatrix}.
\]
so that the intertwining operator can be written as
\[
(M(s)f)(g)=\int_{F}f(wn(u)g)\rmd u.
\]

We define
\[
\varphi(x)=f(r(x))
\]
for $f\in \cH_s$. We first consider the behavior of this function as $|x|\to +\infty$. By direct computation we have
\[
r(x)=\overline n(x)\begin{pmatrix}
  0 & -1 \\
  1 & 0
\end{pmatrix}=
\begin{pmatrix}
  x^{-1} & 1 \\
  0 & x
\end{pmatrix}
\begin{pmatrix}
  -1 & 0 \\
  x^{-1} & -1
\end{pmatrix}
\]
for $|x|>1$. Hence 
\begin{equation}\label{eq:translationborel}
\varphi(x)=f\left(\begin{pmatrix}
  x^{-1} & 1 \\
  0 & x
\end{pmatrix}
\begin{pmatrix}
  -1 & 0 \\
  x^{-1} & -1
\end{pmatrix}\right)=|x|^{-2s-1}f
\begin{pmatrix}
  -1 & 0 \\
  x^{-1} & -1
\end{pmatrix}.
\end{equation}
This means that $\varphi\in \cH_s^{\mathrm{nc}}$, where
\[
\index{hsnc@$\cH_s^{\mathrm{nc}}$}\cH_s^{\mathrm{nc}}:=\{\varphi:F\to\mathbb C\,|\,\rho^{1+2s}\varphi\text{ extends locally constantly across }\infty\}
\]
and \index{rhox@$\rho(x)$}$\rho(x)=\max\{1,|x|\}$. Hence we obtain a map \index{vs@$V_s$}$V_s:\cH_s\to \cH_s^{\mathrm{nc}}$. 

We claim that it is an isomorphism. Suppose that $V_s(f_1)=V_s(f_2)$. Then $f_1$ and $f_2$ agree on elements of the form
\[
\begin{pmatrix}  a & b \\  0 & d\end{pmatrix}\begin{pmatrix}  0 & -1 \\  1 & x^{-1}\end{pmatrix} =\begin{pmatrix}  a & b \\  0 & d\end{pmatrix}\begin{pmatrix}  1 & 0 \\  x & 1 \end{pmatrix}\begin{pmatrix}
  0 & -1 \\
  1 & 0
\end{pmatrix}
\]
such elements form a dense subset of $G$. Hence $f_1=f_2$.
For the surjectivity of $V_s$, we first note that 
\[
\left\{\begin{pmatrix}  0 & -1 \\  1 & -x\end{pmatrix}\,\middle|\,x\in F\right\}\sqcup \left\{\begin{pmatrix}  1 & 0 \\  0 & 1\end{pmatrix}\right\}
\]
form a complete set of representatives of $B\bs G$. Hence we can define a function $f$ on $G$ that satisfies the desired $B$-transformation law and is locally constant near $r(x)$. Since $\varphi\in \cH_s^{\mathrm{nc}}$, by \eqref{eq:translationborel} the function is locally constant near the identity element. Hence $f\in \cH_s$.

For $u\neq 0$, by direct computation we have
\[
wn(u)r(x)=
\begin{pmatrix}
-1&x\\
u&-1-ux
\end{pmatrix}=
\begin{pmatrix}
u^{-1}&-1\\
0&u
\end{pmatrix}
r(x+u^{-1}).
\]
Hence
\[
(M(s)f)(r(x))=\int_F f(wn(u)r(x))\rmd u=\int_F|u|^{-2s-1}f(r(x+u^{-1}))\rmd u.
\]
By making the change of variable $u\to u^{-1}$, we obtain
\[
(M(s)f)(r(x))=\int_F|u|^{2s-1}f(r(x+u))\rmd u=\int_F|u-x|^{2s-1}f(r(u))\rmd u.
\]
and thus
\begin{equation}\label{eq:noncompactintertwining}
(R(s)f)(r(x))=\int_F|u|^{2s-1}f(r(x+u))\rmd u=a(s)\int_F|u-x|^{2s-1}f(r(u))\rmd u,
\end{equation}
where \index{as@$a(s)$}$a(s)=m(s)^{-1}$.

Let \index{rsnc@$R_{\mathrm{nc}}(s)$}$R_{\mathrm{nc}}(s):\cH_s^{\mathrm{nc}}\to \cH_{-s}^{\mathrm{nc}}$ be the operator 
\[
  \varphi\mapsto
  \left(x\mapsto
  a(s)\int_F |x-y|^{-1+2s}\varphi(y)\rmd y\right),
\]
which converges when $\Re s>0$. Then by \eqref{eq:noncompactintertwining} we obtain a commutative diagram
\[
\begin{tikzcd}
	{\mathcal{H}_s} && {\mathcal{H}_{s}^\mathrm{nc}} \\
	\\
	{\mathcal{H}_{-s}} && {\mathcal{H}_{-s}^\mathrm{nc}}
	\arrow["{V_s}"', from=1-1, to=1-3]
	\arrow["{R(s)}", from=1-1, to=3-1]
	\arrow["{R_{\mathrm{nc}}(s)}"', from=1-3, to=3-3]
	\arrow["{V_{-s}}", from=3-1, to=3-3]
\end{tikzcd}
\]
for $\Re s>0$. Then we define
\[
 \index{us@$U_s$} U_s:\mathcal H_{\mathrm c}\rightarrow\mathcal H_s^{\mathrm{nc}},
  \qquad
  (U_sf)(x)=C^{1/2}\rho(x)^{-1-2s}f(x),
\]
which is clearly an isomorphism.

We define the measure on \index{pf@$\bP^1(F)$}$X:=\bP^1(F)$ to be
\begin{equation}\label{eq:boundary-measure}
  \rmd\mu(x)=C\frac{\rmd x}{\rho(x)^2}.
\end{equation}
via the canonical inclusion $F\hookrightarrow \bP^1(F)$, where \index{c@$C$}$C=q/(q+1)$. We will see later that the constant $C$ is chosen so that $(X,\mu)$ is a probability space.

A direct substitution, using $\rmd y=C^{-1}\rho(y)^2\rmd\mu(y)$, gives
\[
U_{-s}^{-1}\left(x\mapsto\int_F |x-y|^{-1+2s}(U_s f)(y)\rmd y\right)(x)=C^{-1}\int_X\delta(x,y)^{-1+2s}f(y)\rmd\mu(y),
\]
where the "distance" on $X$ is defined by
\begin{equation}\label{eq:chordal}
\index{deltaxy@$\delta(x,y)$}\delta(x,y)=\frac{|x-y|}{\rho(x)\rho(y)}
\end{equation}
for $x,y\in F$, and $\delta(x,\infty)=\delta(\infty,x)=1/\rho(x)$ for $x\in F$, and $\delta(\infty,\infty)=0$. $\delta$ defined in this way is continuous on $X\times X$ since $\delta$ is a metric on $X$ and the metric induces the initial topology on $X$ by considering neighborhoods of each point of $X$.

If we define
\[
\index{rsc@$R_{\mathrm c}(s)$}R_{\mathrm{c}}(s):\cH_{\mathrm c}\to \cH_{\mathrm c}
\]
to be
\begin{equation}\label{eq:compact-intertwiner}
(R_{\mathrm{c}}(s)f)(x)=c(s)\int_X\delta(x,y)^{-1+2s}f(y)\rmd\mu(y),
\end{equation}
where \index{cs@$c(s)$}$c(s)=C^{-1}a(s)$, then we obtain a commutative diagram
\[
\begin{tikzcd}
	{\mathcal{H}_s^{\mathrm{nc}}} && {\mathcal{H}_\mathrm{c}} \\
	\\
	{\mathcal{H}_{-s}^{\mathrm{nc}}} && {\mathcal{H}_\mathrm{c}}
	\arrow["{U_s^{-1}}"', from=1-1, to=1-3]
	\arrow["{R(s)}", from=1-1, to=3-1]
	\arrow["{R_{\mathrm c}(s)}"', from=1-3, to=3-3]
	\arrow["{U_{-s}^{-1}}", from=3-1, to=3-3]
\end{tikzcd}.
\]
for $\Re s>0$. Finally we define \index{ws@$W_s$}$W_s:=U_s^{-1}\circ V_s$ and obtain the desired normalized intertwining operator $R_{\mathrm{c}}(s)$ for $\Re s>0$. If $\Re s\leq 0$, we use analytic continuation of these operators and the commutativity of the diagrams to define $R_{\mathrm c}(s)$.

\begin{lemma}\label{lem:boundary-measure}
The measure \eqref{eq:boundary-measure} has total volume $1$ on $X$ and is $K$-invariant.  If $g=\left(\begin{smallmatrix}a&b\\c&d\end{smallmatrix}\right)$ acts by $gx=(ax+b)/(cx+d)$, then
\[\index{jgx@$j_g(x)$}
  j_g(x):=\frac{\rmd\mu(gx)}{\rmd\mu(x)}
  =\frac{\mathopen{|}\det g\mathclose{|}\rho(x)^2}{|cx+d|^2\rho(gx)^2}.
\]
In particular, for $g\in K$ one has $j_g=1$. The representation of $G$ on $\cH_{\mathrm c}$ with respect to $s=0$ is the half-density action
\[
  (\pi_0(g)f)(x)=j_{g^{-1}}(x)^{1/2}f(g^{-1}x).
\]
Equivalently, $(\pi_0(g^{-1})f)(x)=j_g(x)^{1/2}f(gx)$. 
\end{lemma}

\begin{proof}
The part of $F$ with $|x|\leq1$ contributes $C$.  For $n\geq1$, the shell
$|x|=q^n$ has additive measure $(1-q^{-1})q^n$ and contributes
$C(1-q^{-1})q^{-n}$.  Hence
\[
 \mu(X)=C\left(1+(1-q^{-1})\sum_{n\geq1}q^{-n}\right)
 =C(1+q^{-1})=1.
\]
The fractional-linear change of variables has $\rmd(gx)=\mathopen{|}\det g\mathclose{|}\,|cx+d|^{-2}\rmd x$, which gives the displayed formula for $j_g$.  If $g\in K$, preservation of the sup norm on $F^2$ gives
\[
 \max\{|ax+b|,|cx+d|\}=\max\{|x|,1\}=\rho(x),
\]
and therefore $\rho(gx)=\rho(x)/|cx+d|$ and $j_g=1$. 

To prove the last formula, we need to show the following diagram
\[
\begin{tikzcd}
	{\mathcal{H}_0} && {\mathcal{H}_\mathrm{c}} \\
	\\
	{\mathcal{H}_{0}} && {\mathcal{H}_\mathrm{c}}
	\arrow["{W_0}"', from=1-1, to=1-3]
	\arrow["{\pi_0(g)}", from=1-1, to=3-1]
	\arrow["{\pi_0(g)}"', from=1-3, to=3-3]
	\arrow["{W_{0}}", from=3-1, to=3-3]
\end{tikzcd}
\]
commutes. By the construction of this diagram, it suffices to show the commutativity of the diagram 
\[
\begin{tikzcd}
	{\mathcal H_0} && {\mathcal H_0^{\mathrm{nc}}} && {\mathcal H_{\mathrm c}} \\
	\\
	{\mathcal H_0} && {\mathcal H_0^{\mathrm{nc}}} && {\mathcal H_{\mathrm c}}
	\arrow["{V_0}"', from=1-1, to=1-3]
	\arrow["{\pi_0(g)}"', from=1-1, to=3-1]
	\arrow["{U_0^{-1}}"', from=1-3, to=1-5]
	\arrow["{\pi_0(g)}"', from=1-3, to=3-3]
	\arrow["{\pi_0(g)}"', from=1-5, to=3-5]
	\arrow["{V_0}"', from=3-1, to=3-3]
	\arrow["{U_0^{-1}}"', from=3-3, to=3-5]
\end{tikzcd},
\]
where $\pi_0(g)$ on $\mathcal H_0^{\mathrm{nc}}$ is defined by
\[
(\pi_0(g)\varphi)(x)=\frac{\mathopen{|}\det g\mathclose{|}^{1/2}}{|a-cx|}\varphi(g^{-1}x)
\]
for $g=(\begin{smallmatrix} a & b \\ c & d \end{smallmatrix})$, which can be checked directly by using the identity
\[
r(x)g=\begin{pmatrix}
        \dfrac{\det g}{a-cx} & -c \\
        0 & a-cx
      \end{pmatrix}r(g^{-1}x)
\]
for $x\in F$.
\end{proof}

The next proposition gives a method to compute the trace.
\begin{proposition}\label{prop:graph-trace}
Let $A$ be an integral operator on the measure space $(X,\mu)$ whose kernel $K_A(x,y)$ is locally integrable away from the diagonal.  Let $g\in \G(F)$ be such that the pullback $x\mapsto K_A(g^{-1}x,x)$ exists as a locally integrable function. 
Then its distributional density is
\[
  H_A(g)=\int_XK_A(g^{-1}x,x)j_{g^{-1}}(x)^{1/2}\rmd\mu(x),
\]
whenever the integral makes sense.  More generally, if a test function $h\in C_c^\infty(\G(F))$ is supported where this convergence is locally uniform, then
\begin{equation}\label{eq:smearedtrace}
  \Tr\bigl(A\pi_0(h)\bigr)=\int_{\G(F)}h(g) H_A(g)\rmd g.
\end{equation}
\end{proposition}

\begin{proof}
For a test function $f$ in the representation space, by \autoref{lem:boundary-measure} we have
\[
 (A\pi_0(g)f)(z)
 =\int_XK_A(z,x)j_{g^{-1}}(x)^{1/2}f(g^{-1}x)\rmd\mu(x).
\]
Make the change of variables $y=g^{-1}x$.  Since $\rmd\mu(y)=j_{g^{-1}}(x)\rmd\mu(x)$, the kernel of the composed operator at $(z,y)$ is
\[
 K_A(z,gy)j_{g^{-1}}(gy)^{-1/2}.
\]
Restricting it to $z=y$ and then substituting $y=g^{-1}x$ gives exactly $H_A(g)$.  For $h$ locally constant and compactly supported, choose a compact open subgroup $J$ such that $h$ is bi-$J$-invariant. Then $\pi_0(h)$ maps into the finite-dimensional space of $J$-fixed vectors (equivalently, $B\backslash G/J$ is finite on the support), so it has finite rank.  Absolute convergence and Fubini's theorem now justify \eqref{eq:smearedtrace}.  The same proof applies to a meromorphic kernel in a half-plane of absolute convergence and then, vector by vector on the finite-dimensional range, by meromorphic continuation.
\end{proof}

\subsection{Direct calculation for the split case}\label{subsec:split}
Let $t=\operatorname{diag}(a,b)\in A$.  The action of $t^{-1}$ on $x$ is
\[
t^{-1}x=\frac{a^{-1}x}{b^{-1}}=rx,
\]
where $r=b/a$. By \autoref{lem:boundary-measure} the Radon-Nikodym factor is
\[
  j_{t^{-1}}(x)=\frac{|a^{-1}b^{-1}|}{|b^{-1}|^2}\frac{\rho(x)^2}{\rho(rx)^2}=|r|\frac{\rho(x)^2}{\rho(rx)^2}.
\]
Initially, the \emph{meromorphically continued trace density} is defined by
\[\index{psist@$\Psi_s(t)$}
  \Psi_s(t)=|D(t)|^{1/2}c(s)\int_X\delta(x,rx)^{-1+2s}j_{t^{-1}}(x)^{1/2}\rmd\mu(x).
\]
We will see later that it defines a holomorphic function on the half plane $\Re s>0$.

 By \autoref{prop:graph-trace}, this is the distributional trace density of $R_{\mathrm c}(s)\pi_0(g_\theta)$.  The representation here is the fixed representation $\pi_0$, not $\pi_s$: Arthur differentiates the intertwining family and then composes it with $\pi_0(g)$.
Since
\[
  |D(t)|^{1/2}=|1-r||r|^{-1/2},
\]
and
\[
 \delta(x,rx)=\frac{|1-r||x|}{\rho(x)\rho(rx)},
\]
the definitions of $\delta$, $j_{t^{-1}}$, and $\mu$ give, factor by factor, the exact cancellation
\[
|D(t)|^{1/2}
  \delta(x,rx)^{-1}j_{t^{-1}}(x)^{1/2}\rmd\mu(x)=|1-r||r|^{-1/2}
  \frac{\rho(x)\rho(rx)}{|1-r||x|}
  \left(|r|\frac{\rho(x)^2}{\rho(rx)^2}\right)^{1/2}
  C\frac{\rmd x}{\rho(x)^2}=C\frac{\rmd x}{|x|}.
\]
It follows that
\[
  \Psi_s(t)=c(s)C\int_{F^\times}
  \delta(x,rx)^{2s}\frac{\rmd x}{|x|}.
\]
Hence $\Psi_s(t)$ is well defined on the half plane $\Re s>0$ and defines a holomorphic function there. 

For $g_\theta$ in \autoref{lem:local-section}, by \autoref{prop:graph-trace} and the Weyl integration formula (since $g_\theta$ is supported on split elements), we have
\begin{equation}\label{eq:localtraceidentity}
  \Tr\bigl(R_{\mathrm c}(s)\pi_0(g_\theta)\bigr)
  =\frac12\int_A\phi_\theta(t)\Psi_s(t)\rmd t
  =\int_U\theta(t)\Psi_s(t)\rmd t.
\end{equation}
for $s$ such that $\Re s>0$. 

Set
\[
  u=v(r),\qquad m=v(1-r).
\]
If $v(x)=k$, then $\rho(x)=q^{-\min\{0,k\}}$ and $\rho(rx)=q^{-\min\{0,u+k\}}$.  Since $|1-r||x|=q^{-(m+k)}$, direct substitution gives
\begin{equation}\label{eq:ek}
  e_k=m+k-\min\{0,k\}-\min\{0,u+k\}.
\end{equation}
Equivalently, $\delta(x,rx)=q^{-e_k}$.  The additive measure of every valuation shell, after division by $|x|$, is $1-q^{-1}$.  Therefore
\begin{equation}\label{eq:Psi-sum}
  \Psi_s(t)=c(s)C(1-q^{-1})\sum_{k\in\mathbb Z}z^{e_k}
  =\frac{(q-1)(1-z)}{q-z}\sum_{k\in\mathbb Z}z^{e_k},
\end{equation}
where \index{zs@$z(s)$}$z=z(s)=q^{-2s}$.

\begin{proposition}\label{prop:psicomputation}
For every $t\in \A(F)_\reg$, $\Psi_s$ can be analytic continued to $s=0$ with
\begin{equation}\label{eq:Psi-derivative}
  \Psi_0(t)=2,\qquad\Psi_0'(t)=-2\frac{q+1}{q-1}\log q+2\log|D(t)|.
\end{equation}
\end{proposition}

\begin{proof}
Assume first that $u>0$.  Then $m=0$.  The exponent in \eqref{eq:ek} is the distance from $k$ to the integer interval $[-u,0]$,
and hence
\[
  \sum_{k\in\mathbb Z}z^{e_k}=u+1+\frac{2z}{1-z}.
\]
Substitution in \eqref{eq:Psi-sum} gives
\[
  \Psi_s(t)=
  (q-1)\frac{u+1+(1-u)z}{q-z}.
\]
Since $z(0)=1$ and $z'(0)=-2\log q$,
\[
  \Psi_0'(t)
  =-2\log q\left(1-u+\frac{2}{q-1}\right).
\]
In this case $|D(t)|=|1-r|^2|r|^{-1}$, so that $\log|D(t)|=u\log q$, so this is precisely \eqref{eq:Psi-derivative}.

Suppose next that $u=0$.  Then
\[
  \sum_{k\in\mathbb Z}z^{e_k}=z^m\frac{1+z}{1-z},
\]
and therefore
\[
  \Psi_s(t)=(q-1)\frac{z^m(1+z)}{q-z}.
\]
Differentiating at zero gives
\[
  \Psi_0'(t)
  =-2\frac{q+1}{q-1}\log q-4m\log q.
\]
But $\log|D(t)|=-2m\log q$, proving the formula.  Finally, if $u<0$, then
\[
  \operatorname{diag}(a,b)
  =w\operatorname{diag}(b,a)w^{-1}.
\]
Since $w\in K$, the operator $R_{\mathrm c}(s)$ commutes with $\pi_0(w)$, and $\pi_0$ has trivial central character.  Consequently $\Psi_s(\operatorname{diag}(a,b))=\Psi_s(\operatorname{diag}(b,a))$.  Replacing $\operatorname{diag}(a,b)$ by $\operatorname{diag}(b,a)$ (so that $r$ becomes $r^{-1}$) therefore reduces the case $u<0$ to the already computed case $-u>0$; the Weyl discriminant is unchanged.
\end{proof}

Differentiating the local trace identity \eqref{eq:localtraceidentity} at $s=0$ gives
\[
  \Tr\bigl(R'(0)\pi_0(g_\theta)\bigr)
  =\int_U\theta(t)\Psi'_0(t)\rmd t.
\]
Since $R(0)$ is the identity, we find that
\[
  J_\A(\triv,g_\theta)
  =\int_U\theta(t)\Psi'_0(t)\rmd t.
\]

For $g_\theta$ from \autoref{lem:local-section}, the dual invariantization \cite[(3.5)$\spcheck$]{arthur1994} is
\[
  I_\A(\triv,g_\theta)
  =J_\A(\triv,g_\theta)-\widehat I_\A^\A(\triv,\phi_\A(g_\theta))=J_\A(\triv,g_\theta).
\]
Now apply Arthur's equation \eqref{eq:Arthur65}.  For $\M=\L=\A$ the parity is even, $i^\A(\triv)=1$, and $\triv^\vee=\triv$. Therefore,
\[
  I_\A(\theta,\triv)=\int_U\theta(t)\Psi'_0(t)\rmd t.
\]
This identity holds for every $\theta\in C_c^\infty(U)$.  Apply \autoref{thm:smeared-to-pointwise} with $F(t)=\Psi'_0(t)$.  It identifies that function pointwise with Arthur's reconstruction.  Hence, on $U$,
\[
  I_{\A,\A}^\G(t,\triv)=\Psi'_0(t).
\]
Substitution of \eqref{eq:Psi-derivative} proves
\begin{equation}\label{eq:split2}
  I_{\A,\A}^\G(t,\triv)  =-2\frac{q+1}{q-1}\log q  +2\log|D(t)|.
\end{equation}

\subsection{Direct calculation for the elliptic case}
For later comparison with the compact computation, note explicitly that
\[
  a(0)=0,\qquad a'(0)=\frac{2\log q}{1-q^{-1}}.
\]

As in the split case, we define the meromorphically continued trace density as
\[\index{phisgamma@$\Phi_s(\gamma)$}
\Phi_s(\gamma)=|D(\gamma)|^{1/2}c(s)\int_X\delta(\gamma^{-1}x,x)^{-1+2s}j_{\gamma^{-1}}(x)^{1/2}\rmd\mu(x).
\]
We will see that it is absolutely convergent for any $s$ and defines an entire function.

Let
\[
  \gamma=\begin{pmatrix}a&b\\c&d\end{pmatrix},
  \qquad
 \index{deltagamma@$\Delta_\gamma$} \Delta_\gamma=(a+d)^2-4\det\gamma,
\]
and define
\[
\index{pgammax@$P_\gamma(x)$} P_\gamma(x)=cx^2+(d-a)x-b.
\]
In the space $(X,\mu)$ at the trivial datum one has
\[
 (\pi_0(\gamma)f)(x)
 =\bigl(j_{\gamma^{-1}}(x)\bigr)^{1/2}f(\gamma^{-1}x).
\]
For the fractional linear action occurring here,
\[
  \gamma^{-1}x=\frac{dx-b}{-cx+a},
\]
we have
\[
  x-\gamma^{-1}x=-\frac{P_\gamma(x)}{-cx+a},
\]
so $P_\gamma$ is the fixed-point polynomial, and
\[
  \disc(P_\gamma)=\Delta_\gamma.
\]
If $\gamma$ is elliptic, then $c\neq0$ (otherwise $\gamma$ is triangular) and $P_\gamma$ is an irreducible quadratic, so it has no zero in $F$. The affine chart omits only one measure-zero point.  The formulas below hold away from the further measure-zero point $-cx+a=0$.  Combining the fractional-linear change of variables with $\rmd\mu=C\rho^{-2}\rmd x$ gives the factor
\[
  j_{\gamma^{-1}}(x)
  =\frac{\mathopen{|}\det\gamma\mathclose{|}\rho(x)^2}
         {\mathopen{|}-cx+a\mathclose{|}^2\rho(\gamma^{-1}x)^2}.
\]
Indeed, the derivative with respect to additive Haar measure is $\mathopen{|}\det\gamma\mathclose{|}/\mathopen{|}-cx+a\mathclose{|}^2$; the extra ratio of $\rho$-factors changes it to the derivative with respect to $\mu$.  From the definitions of $\delta$, $j$, and $\mu$ we now get, with every factor displayed,
\begin{align*}
 \delta(\gamma^{-1}x,x)^{-1}
  \bigl(j_{\gamma^{-1}}(x)\bigr)^{1/2}\rmd\mu(x)&=
  \frac{\mathopen{|}-cx+a\mathclose{|}\rho(\gamma^{-1}x)\rho(x)}{|P_\gamma(x)|}
  \frac{\mathopen{|}\det\gamma\mathclose{|}^{1/2}\rho(x)}
       {\mathopen{|}-cx+a\mathclose{|}\rho(\gamma^{-1}x)}
  C\frac{\rmd x}{\rho(x)^2} \\
 &=C\mathopen{|}\det\gamma\mathclose{|}^{1/2}
  \frac{\rmd x}{|P_\gamma(x)|}.
\end{align*}

The coefficient of Arthur's normalized orbital integral is obtained by multiplying by $|D(\gamma)|^{1/2}$.  If $\lambda_1,\lambda_2$ are the eigenvalues of $\gamma$, then
\[
  D(\gamma)=\left(1-\frac{\lambda_1}{\lambda_2}\right) \left(1-\frac{\lambda_2}{\lambda_1}\right) =-\frac{\Delta_\gamma}{\det\gamma}.
\]
Therefore
\[
  |D(\gamma)|^{1/2}\mathopen{|}\det\gamma\mathclose{|}^{1/2}=|\Delta_\gamma|^{1/2}.
\]
Hence we see that (recall that $a(s)=Cc(s)$)
\[
\Phi_s(\gamma)=a(s)|\Delta_\gamma|^{1/2}\int_F\delta(\gamma^{-1}x,x)^{2s}\frac{\rmd x}{|P_\gamma(x)|}.
\]
The integral converges uniformly on every compact set of $G_\el$. Indeed, $P_\gamma$ has no zero in $F$, so its reciprocal is locally bounded, and $|P_\gamma(x)|\asymp|x|^2$ for large $|x|$. Moreover, by the definition of $\delta(x,y)$ it is easy to see that $\delta(x,y)\leq 1$. Since $\gamma$ is elliptic, $\gamma^{-1}x\neq x$ for all $x\in X$. Since $X$ is compact and $\delta$ is continuous on $X\times X$, we find that
\[
\inf_{x\in X,\, \gamma\in \cC}\delta(\gamma^{-1}x,x)>0
\]
for every compact subset $\cC$ of $G_\el$. Hence we find that $\Phi_s(\gamma)$ is entire in $s$.

By \autoref{prop:graph-trace} and the Weyl integration formula (since $g_\theta$ is supported on conjugacy classes of a fixed elliptic torus $\T$), we have
\begin{equation}\label{eq:localtraceidentity2}
  \Tr\bigl(R(s)\pi_0(g_\theta)\bigr)
  =\frac12\int_T\phi_\theta(t)\Phi_s(t)\rmd t
  =\int_U\theta(t)\Phi_s(t)\rmd t.
\end{equation}

Differentiating \eqref{eq:localtraceidentity2} at $s=0$, we find that
\[
 I_\A(\triv,g_\theta)=J_\A(\triv,g_\theta)=\Tr\bigl(R_{\mathrm c}'(0)\pi_0(g_\theta)\bigr)
=\int_U\theta(t)\Phi_0'(t)\rmd t,
\]
where
\[
\Phi_0'(\gamma)=a'(0)|\Delta_\gamma|^{1/2}\int_F\frac{\rmd x}{|P_\gamma(x)|}.
\]
Therefore by \eqref{eq:Arthur65} and \autoref{thm:smeared-to-pointwise},
\begin{equation}\label{eq:elliptic-density-integral}
I_{\A,\G}^\G(\gamma,\triv)=-\frac{2\log q}{1-q^{-1}}|\Delta_\gamma|^{1/2}\int_F\frac{\rmd x}{|P_\gamma(x)|}.
\end{equation} 

Finally, we evaluate the elementary quadratic integral in \eqref{eq:elliptic-density-integral}.

\begin{lemma}\label{lem:quadratic-integral}
Let $P\in F[x]$ be an irreducible quadratic polynomial with discriminant $\Delta$ and let $E=F(\sqrt\Delta)$.  Then
\begin{equation}\label{eq:quadratic-integral}
  |\Delta|^{1/2}\int_F\frac{\rmd x}{|P(x)|}
  =
  \begin{cases}
  1+q^{-1},&E/F\text{ unramified},\\[2mm]
  2q^{-\delta(E/F)/2},&E/F\text{ ramified},
  \end{cases}
\end{equation}
where $\delta(E_\gamma/F)$ is the exponent of the discriminant ideal. 
\end{lemma}
\begin{proof}
We use only the standard facts that an unramified quadratic extension has an integral generator with irreducible reduction, that a uniformizer of a totally ramified extension has Eisenstein minimal polynomial, and that the discriminant of an integral basis generates the discriminant ideal. These facts are in \cite[Chapter~III, \SSec\SSec 3--6, pp.~50--58]{serre1979}.

The expression on the left is invariant under
\[
  P(x)\mapsto uP(vx+w),
  \qquad u,v\in F^\times,\quad w\in F.
\]
Indeed, the square root of the discriminant is multiplied by $|uv|$, while the change of variables in the integral multiplies it by $|u|^{-1}|v|^{-1}$.  We may therefore replace a root of $P$ by any affine $F$-linear generator of the same quadratic extension.

If $E/F$ is unramified, choose an integral generator whose monic minimal polynomial has irreducible reduction.  Then $|P(x)|=1$ for every $x\in\cO$.  Since $P$ is monic and integral, $|P(x)|=|x|^2$ for $|x|>1$.  Thus
\[
  \int_\cO\frac{\rmd x}{|P(x)|}=1,\qquad\int_{F-\cO}\frac{\rmd x}{|P(x)|}=(1-q^{-1})\sum_{n\geq1}q^{-n}=q^{-1}.
\]
The discriminant is a unit, giving the first line of \eqref{eq:quadratic-integral}.

If $E/F$ is ramified, choose a uniformizer $\varpi_E$ of $E$.  Since the residue extension is trivial and $[E:F]=2$, one has $\cO_E=\cO[\varpi_E]$, and its monic minimal polynomial is Eisenstein. On $\cO^\times$ the polynomial is a unit, while on $\varpi\cO$ its absolute value is $q^{-1}$.  Hence
\[
  \int_{\cO^\times}\frac{\rmd x}{|P(x)|}=1-q^{-1},
  \qquad
  \int_{\varpi\cO}\frac{\rmd x}{|P(x)|}=1.
\]
The integral outside $\cO$ is again $q^{-1}$, so the total integral is $2$.  The discriminant of this integral basis has absolute value $q^{-\delta(E/F)}$, proving the second line.  
\end{proof}
By \eqref{eq:elliptic-density-integral} and the above lemma we obtain
\begin{equation}\label{eq:elliptic2}
  I_{\A,\G}^{\G}(\gamma,\triv)=
  \begin{dcases}
  -2\frac{1+q^{-1}}{1-q^{-1}}\log q,
       &E_\gamma/F\text{ unramified},\\
  -\frac{4q^{-\delta(E_\gamma/F)/2}}{1-q^{-1}}\log q,
       &E_\gamma/F\text{ ramified}.
  \end{dcases}
\end{equation}

\subsection{Proof of \autoref{thm:raw-main}}
It remains only to pass from $I_{\L,\M}(\gamma,\tau)$ to $I_{\L,\M}(\tau,\gamma)$. Here $i^\A(\triv)=1$ and $\triv^\vee=\triv$. Arthur's reciprocity formula (4.6) reads
\begin{equation}\label{eq:reciprocity}
  I_{\A,\M}^\G(\gamma,\triv)
  =(-1)^{\dim(A_\M\times A_\A)}
    I^{\G}_{\A,\M}(\triv,\gamma).
\end{equation}
Now $\dim A_\A=2$ and $\dim A_\G=1$.  The sign is positive for $\M=\A$ and negative for $\M=\G$.  Combining \eqref{eq:reciprocity} with \eqref{eq:split2} (resp. \eqref{eq:elliptic2})  proves \eqref{eq:raw-split} (resp. \eqref{eq:raw-elliptic}).  Finally, \eqref{eq:torus-kernel} follows immediately because $\A$ is abelian and there is no proper-Levi subtraction in $\A$.

\subsection{Specialization to $\QQ_p$}\label{subsec:explicitqp}
We now put $F=\QQ_p$ so that $q=p$. The values are exactly the desired ones:

\subsubsection{Odd primes}
For $p\neq2$, a ramified quadratic extension has discriminant exponent one.
Thus
\[
  I_{\A,\G}^{\G}(\triv,\gamma)=
  \begin{dcases}
  \frac{4p^{-1/2}}{1-p^{-1}}\log p,
       &E_\gamma/\QQ_p\text{ ramified, i.e., }\omega_p(\gamma)=0,\\
  2\frac{1+p^{-1}}{1-p^{-1}}\log p,
       &E_\gamma/\QQ_p\text{ unramified, i.e., }\omega_p(\gamma)=-1,
  \end{dcases}
\]
and
\[
  I_{\A,\A}^{\G}(\triv,t)=-2\frac{1+p^{-1}}{1-p^{-1}}\log p+2\log|D(t)|_p.
\]

\subsubsection{The prime $p=2$}
For $p=2$, an elliptic quadratic extension is unramified, or ramified with
discriminant exponent $2$ or $3$.  Therefore
\[
  I_{\A,\G}^{\G}(\triv,\gamma)=
  \begin{dcases}
  p^{-1}\frac{4}{1-p^{-1}}\log p=4\log2,&\delta(E_\gamma/\QQ_2)=2, \text{ i.e., }\omega_p(\gamma)=0,\ \lambda_p(\gamma)=0,\\
  p^{-3/2}\frac{4}{1-p^{-1}}\log p=2\sqrt2\log2,&\delta(E_\gamma/\QQ_2)=3,\text{ i.e., }\omega_p(\gamma)=0,\ \lambda_p(\gamma)=1,\\
  2\frac{1+p^{-1}}{1-p^{-1}}\log p=6\log2,&E_\gamma/\QQ_2\text{ unramified, i.e., }\omega_p(\gamma)=-1,
  \end{dcases},               
\]
and
\[
  I_{\A,\A}^{\G}(\triv,t)=-2\frac{1+p^{-1}}{1-p^{-1}}\log p+2\log|D(t)|_p=-6\log2+2\log|D(t)|_2.
\]

\bibliography{ref.bib}
\bibliographystyle{amsalpha}

\printindex

\end{document}